\documentclass[11pt,reqno,a4paper]{amsart}
\usepackage{amsfonts}
\usepackage{ifthen}
\usepackage{amsthm,cite}
\usepackage{amsmath,mathrsfs}
\usepackage{graphicx}
\usepackage{amscd,amssymb,amsthm}
\usepackage{color}
\usepackage{hyperref}
\usepackage{amsthm}
\usepackage{amsmath}

\allowdisplaybreaks

\usepackage{array,multirow,makecell}

\setcellgapes{1pt}
\makegapedcells
\newcolumntype{R}[1]{>{\raggedleft\arraybackslash }b{#1}}
\newcolumntype{L}[1]{>{\raggedright\arraybackslash }b{#1}}
\newcolumntype{C}[1]{>{\centering\arraybackslash }b{#1}}
\newcounter{minutes}
\divide\time by 60
\newcounter{hours}
\multiply\time by 60 \addtocounter{minutes}{-\time}

\newtheorem{theo}{Theorem}[section]

\newtheorem{theorem}{Theorem}[section]

\newtheorem{lemma}[theorem]{Lemma}
\newtheorem{definition}[theorem]{Definition}
\newtheorem{corollary}[theorem]{Corollary}
\newtheorem{remark}[theorem]{Remark}
\newtheorem{example}[theorem]{Example}

\title[]{On the Bessel Function and $n$-dimensional Hankel transform with Bicomplex arguments and coherent states}
\author[Snehasis Bera]{Snehasis Bera}
\address{Snehasis Bera\newline Department of Mathematics,\newline National Institute of Technology Jamshedpur, Jamshedpur-831014, Jharkhand, India.}
\email{berasnehasis1996@gmail.com}
\author[Sourav Das]{Sourav Das$^\ast$}
\thanks{$^\ast$Corresponding author}
\address{Sourav Das\newline Department of Mathematics,\newline National Institute of Technology Jamshedpur, Jamshedpur-831014, Jharkhand, India.}
\email{souravdasmath@gmail.com, souravdas.math@nitjsr.ac.in}
\author[Abhijit Banerjee]{Abhijit Banerjee}
\address{Abhijit Banerjee\newline Department of Mathematics,\newline Garhbeta College, Paschim Medinipur-721127, West Bengal, India.}
\email{abhijit.banerjee.81@gmail.com}
\date{}
\begin{document}

\begin{abstract}
  In this work, we introduce bicomplex Bessel function and analyze its region of convergence. Important properties of the bicomplex Bessel function, such as recurrence relations, integral representations, differential relations are explored. Moreover a differential equation satisfied by the bicomplex Bessel function is established. Furthermore, we investigate bicomplex holomorphicity and discuss its asymptotic behavior. Finally, we define $n$-dimensional bicomplex Hankel transformation by using bicomplex Bessel function and show  that it is an isomorphism between two suitably defined function spaces. The application of the $n$-dimensional bicomplex Hankel transform has been effectively demonstrated by solving some partial differential equations. Additionally, a new extension of coherent states is built based on the use of the bicomplex Bessel function and demonstrate that these states fulfill the conditions of normalizability, continuity and the resolution of unity.
\end{abstract}
\maketitle
\section{Introduction and motivation}
For several centuries, mathematicians have explored natural extensions of the complex number system and its associated function theory. Among the most prominent extensions are the quaternions, introduced by W. Hamilton in 1843, and the bicomplex numbers, first described by C. Segre \cite{le} in 1892. The set of bicomplex numbers serves as a compelling commutative counterpart to the non commutative skew field of quaternions, another four-dimensional real space. In contrast to quaternions, bicomplex numbers exhibit commutative multiplication and constitute a ring characterized by the presence of zero divisors. Several aspects of bicomplex numbers have been investigated, including their algebraic and geometric properties along with various practical applications \cite{bc_number, bc_holomorphic, multicomplex, hyperbolic}.

Bessel functions play a crucial role in various fields of science and engineering due to their ability to model phenomena with cylindrical or spherical symmetry. In mathematical physics, they are the solutions to Bessel’s differential equation and arise naturally in problems involving wave propagation, heat conduction, and vibrations \cite{transforms}. Bessel functions also appear in quantum mechanics, particularly when solving the Schrödinger equation for systems with central potentials \cite{qu_mechanics}. In engineering, they are used in signal processing, filter design, and control theory. For a more comprehensive discussion on their applications, refer to references \cite{app_bessel, bessel}.

In recent years, there has been a growing focus on extending various functions into bicomplex space, which broadens the scope and applicability of mathematical tools. For example, S. P. Goyal et al. \cite{gamma} extended the notions of Beta and Gamma functions from  the space of complex variables to bicomplex variables, identifying essential properties like the Gauss multiplication theorem, binomial theorem etc. Additionally, R. Meena and A. K. Bhabor \cite{bc_hypergeometric}  further extended the hypergeometric function into the bicomplex domain and derived its recurrence relations and integral representations. Further studies on the bicomplex analogues of special functions and their applications are presented in \cite{hurwitz, zeta}. In the current article it is our aim to obtain a bicomplex version of the well known Bessel function and present some applications of the generalized function.

The collection of bicomplex numbers $\mathbb{BC}$ is represented by \cite{bc_number}: $$\mathbb{BC}=\{Z:Z=\lambda_1+j\lambda_2=a+ib+jc+kd\}$$
where $a,b,c,d\in \mathbb{R},\lambda_1,\lambda_2\in\mathbb{C}$ with $i,j,k$ satisfies the properties $i\cdot i=j\cdot j=-1,k\cdot k=1$ and $i\cdot j=j\cdot i=k$.
For every bicomplex number $Z=\lambda_1+j\lambda_2\in \mathbb{BC} $ there exists a unique representation known as the idempotent representation, given by  $$ Z=(\lambda_1-i\lambda_2)e_1+(\lambda_1+i\lambda_2)e_2=z_1e_1+z_2 e_2,$$ where $e_1=\frac{1+k}{2}$  and \; $e_2=\frac{1-k}{2}$ are two idempotent bicomplex numbers satisfies the identities: $e_1+e_2=1, e_1-e_2=ij, e_1^{2}=e_1, e_2^{2}=e_2, e_1\cdot e_2=e_2\cdot e_1=0$. Three different conjugations exist for any bicomplex number $Z$ on $\mathbb{BC}$ are given by
\begin{align*}
    \overline{Z}=\overline{\lambda_1}+j\overline{\lambda_2},\quad \tilde{Z}=\lambda_1-j\lambda_2,\quad
    Z^*=\overline{\lambda_1}-j\overline{\lambda_2}.
\end{align*}
The space $\mathbb{BC}$ consisting of all bicomplex numbers, forms a commutative ring with unity. The collection of all zero divisors of $\mathbb{BC}$ is defined as \cite{bc_number} : $$\mathcal{O}_2=\{\lambda_1+j\lambda_2:{\lambda_1}^2+{\lambda_2}^2=0\}.$$
A subset $\mathcal{D}$ of $\mathbb{BC}$ is referred to as the set of hyperbolic numbers. The elements of this subset are of the form
 $Z=z_1e_1+z_2e_2$ where $z_1,z_2\in\mathbb{R}$. Inside $\mathcal{D}$ the subsets of non-negative and non-positive hyperbolic numbers can be defined as follows:
\begin{align*}
       \mathcal{D}^+=\{Z=z_1e_1+z_2e_2|z_1,z_2\geq0\}\quad \mbox{and}\quad \mathcal{D}^-=\{Z=z_1e_1+z_2e_2|z_1,z_2\leq0\},
   \end{align*}respectively. Furthermore, within the set $\mathcal{D}$ a partial order relation is defined. Specifically,if $Z,W\in\mathcal{D}$ and $Z<_hW$, it implies that $W-Z\in\mathcal{D}^+$ and $Z-W\in\mathcal{D}^-$.
 The hyperbolic norm $|\cdot|_h$ for any bicomplex number $Z=z_1e_1+z_2 e_2$, is given by\cite{bc_number}: $$|Z|_h=|z_1|e_1+|z_2|e_2,$$ and satisfies multiplicative property $|ZW|_h=|Z|_h|W|_h$ for all $Z,W\in\mathbb{BC}$.

An alternative representation of the bicomplex number $Z=z_1e_1+z_2e_2$ can be given as
\begin{align*}
    \mathcal{S}(Z)=\mathcal{S}_1(Z)e_1+\mathcal{S}_2(Z)e_2,
\end{align*}
where the mappings $\mathcal{S}_1$ and $\mathcal{S}_2$  are projections from $\mathbb{BC}$ to $\mathbb{C}(i)$ defined as follows:
\begin{align*}
    \mathcal{S}_1(Z)=\lambda_1-i\lambda_2\;\; \mbox{and}\;\; \mathcal{S}_2(Z)=\lambda_1+i\lambda_2.
\end{align*}
A bicomplex valued function $g:X\subset \mathbb{BC} \rightarrow \mathbb{BC}$ is differentiable at $Z_0$, if the limit $$\lim\limits_{\substack{{Z\to Z_0}\\{Z-Z_0\notin \mathcal{O}_2}}}\frac{g(Z)-g(Z_0)}{Z-Z_0}$$ exist finitely. Moreover, if $g$ is differentiable at every points in $X$ then it is referred to as bicomplex holomorphic in $X$ and this is equivalent to stating that a bicomplex function $g=g_1+jg_2$ is said to be  holomorphic in bicomplex space iff $g_1$ and $g_2$ are holomorphic function of $\lambda_1$ and $\lambda_2$ and satisfies bicomplex Cauchy-Riemann equations, which are expressed as follows:
\begin{align}\label{eq:i5}
       \frac{\partial g_1}{\partial \lambda_1}= \frac{\partial g_2}{\partial \lambda_2} \quad
      and \quad  \frac{\partial g_1}{\partial \lambda_2}=- \frac{\partial g_2}{\partial \lambda_1}.
      \end{align}

Let $C:\phi(t)=\phi_1(t)e_1+\phi_2(t)e_2,\;a\leq t\leq b$, represent a piecewise continuously differentiable curve in $\mathbb{BC}$ and $C_1:\phi_1(t),\;a\leq t\leq b$, $C_2:\phi_2(t),\;a\leq t\leq b$ are two projection curves in $\mathcal{S}_1,\mathcal{S}_2$ respectively that means $C=(C_1,C_2)$, then integration of bicomplex valued function $g$ is provided by \cite{multicomplex}:
\begin{align*}
    \int_C g(\phi) d\gamma=\left[\int_{C_1}g(\phi_1) d\phi_1\right] e_1+\left[\int_{C_2}g(\phi_2) d\phi_2\right] e_2.
\end{align*}
The bicomplex gamma function is expressed in its Euler product form as follows \cite{gamma}:
    \begin{align*}
        \Gamma_b(Z)=Z^{-1}e^{-\gamma Z}\prod_{k=1}^{\infty}\left(\left(1+\frac{Z}{k}\right)^{-1}\exp\left(\frac{Z}{k}\right)\right),
    \end{align*}
   with $\lambda_1\not=-\frac{(r_1+r_2)}{2}$, $\lambda_2\not=i\frac{(r_2-r_1)}{2}$ where  $r_1,r_2\in \mathbb{N}\cup \{0\}$ and $\gamma$ is the Euler constant \cite{sp_function}. Moreover, the bicomplex gamma function is represented through its idempotent representation as:
    \begin{align}\label{eq:g}
        \Gamma_b(Z)=\Gamma(z_1) e_1+ \Gamma(z_2) e_2.
    \end{align}
 Bessel function of the first kind of order $\nu$ is defined as \cite{bessel} :
    \begin{align}\label{eq:5}
         J_\nu(z)=
  \sum_{n=0}^{\infty}\frac{(-1)^n}{n!\Gamma{(\nu+n+1)}}\left(\frac{z}{2}\right)^{2n+\nu},\; (\nu,z\in \mathbb{C}).
    \end{align}
For our purpose we need to introduce the concept of $\mathcal{D}$-boundedness. We recall from \cite{bc_functional}
\begin{definition}
    The operator $T:X\rightarrow Y$ is called $\mathcal{D}$-bounded if there exists $\kappa\in\mathcal{D}^+$ such that
    \begin{align*}
    |Tx|_h<_h\kappa|x|_h,\;\forall x\in X.
    \end{align*}
\end{definition}
\begin{lemma}\label{lemma2}
    The operator $T:X\rightarrow Y$ is $\mathcal{D}$-bounded if and only if $T$ is continuous.
\end{lemma}
\begin{lemma}\rm{\cite{bicomplex ghf}}
    Let $U(Z)={}_{p} F_{q}\left[\begin{matrix}&\alpha_1, &\alpha_2,.... &\alpha_{p};\\&\beta_1,&\beta_2,.....&\beta_{q};&\end{matrix}Z\right]$ be a bicomplex generalized hypergeometric function with $p\leq q+1$, then $U(Z)$ is a solution of the differential equation
    \begin{align}\label{eq:in}
        \frac{d}{dZ} \prod\limits_{j=1}^{q}\left(Z\frac{d}{dZ}+\beta_j-1\right)U(Z)-\prod\limits_{i=1}^{p}\left(Z\frac{d}{dZ}+\alpha_i\right)U(Z)=0.
    \end{align}
\end{lemma}
\begin{lemma}\rm{\cite{bicomplex ghf}}\label{ls}
     The bicomplex confluent hypergeometric function can be expressed through its integral representation as follows
     \begin{align*}
         {}_1F_{1}(\alpha;\beta;Z)=\frac{\Gamma_b(\beta)}{\Gamma_b(\alpha)\Gamma_b(\beta-\alpha)}\int_{C}t^{\alpha-1}(1-t)^{\beta-\alpha-1}\exp(Zt)dt.
     \end{align*}
\end{lemma}
\section{Bicomplex Bessel function}
This section presents the bicomplex Bessel function, exploring its convergence within bicomplex space. Furthermore, we derive the recurrence relations associated with this function, offering a more comprehensive understanding of its characteristics. We define the bicomplex Bessel function as follows:
\begin{align}\label{eq:gh}
    \mathcal{J}_\mathcal{V}(Z)=
  \sum_{s=0}^{\infty}\frac{(-1)^s}{\Gamma_b{(\mathcal{V}+s+1)}s!}\left(\frac{Z}{2}\right)^{\mathcal{V}+2s},
\end{align}
where $\mathcal{V}=\nu_1e_1+\nu_2e_2,Z=z_1e_1+z_2e_2\in{\mathbb{BC}}$.
The following theorem offers substantial justification for the definition of the bicomplex Bessel function.
\begin{theo}\label{eq:s3}
    Let  $Z=z_1e_1+z_2e_2, \mathcal{V}=\nu_1e_1+\nu_2e_2$ be elements of $\mathbb{BC}$, where $e_1=\frac{1+k}{2}$, $e_2=\frac{1-k}{2}$, then idempotent form of the bicomplex Bessel function is given by
 \begin{align}\label{eq:fa}
     \mathcal{J}_\mathcal{V}(Z)=J_{\nu_1}(z_1)e_1+ J_{\nu_2}(z_2)e_2,
 \end{align}
  and the series $Z^{-\mathcal{V}} \mathcal{J}_\mathcal{V}(Z)$ converges absolutely in the hyperbolic sense for all $Z\in\mathbb{BC}$.
 \end{theo}
 \begin{proof}
 Using the idempotent form of bicomplex gamma function \eqref{eq:g} and \eqref{eq:gh}, we have
\begin{align}
     \mathcal{J}_\mathcal{V}(Z)&=\sum_{s=0}^{\infty}\frac{(-1)^s}{\Gamma_b{(\nu_1e_1+\nu_2e_2+s+1)s!}}\left(\frac{z_1e_1+z_2e_2}{2}\right)^{\nu_1e_1+\nu_2e_2+2s}\nonumber\\&=\sum_{s=0}^{\infty}\frac{(-1)^s}{[\Gamma(\nu_1+s+1)e_1+\Gamma(\nu_2+s+1)e_2]s!}\left[\left(\frac{z_1}{2}\right)e_1+\left(\frac{z_2}{2}\right)e_2\right]^{\nu_1e_1+\nu_2e_2+2s}\nonumber \\
     &=\sum_{s=0}^{\infty}\frac{(-1)^s[\Gamma(\nu_1+s+1)e_2+\Gamma(\nu_2+s+1)e_1]}{\Gamma(\nu_1+s+1)\Gamma(\nu_2+s+1)s!}\left[\left(\frac{z_1}{2}\right)^{\nu_1+2s}e_1+\left(\frac{z_2}{2}\right)^{\nu_2+2s}e_2\right]\nonumber \\
    &=\sum_{s=0}^{\infty}\frac{(-1)^s}{\Gamma{(\nu_1+s+1)s!}}\left(\frac{z_1}{2}\right)^{\nu_1+2s}e_1+\sum_{r=0}^{\infty}\frac{(-1)^r}{\Gamma{(\nu_2+s+1)s!}}\left(\frac{z_2}{2}\right)^{\nu_2+2s}e_2\nonumber\\
    &=J_{\nu_1}(z_1)e_1+ J_{\nu_2}(z_2)e_2.
\end{align}
Since $\nu_1,\nu_2,z_1,z_2\in\mathbb{C}$, then Bessel functions $ J_{\nu_1}(z_1)$ and  $J_ {\nu_2}(z_2)$ are well defined in complex space and bicomplex Bessel functions can be expresse as $\mathcal{J} _ {\mathcal{V}}(Z) =J_{\nu_1}(z_1)e_1+ J_{\nu_2}(z_2)e_2$,  which represents the idempotent form of the bicomplex Bessel function $\mathcal{J} _ {\mathcal{V}}(Z)$. Next, we will examine the convergence of the series $Z^{-\mathcal{V}} \mathcal{J}_{\mathcal{V}}(Z)$ using hyperbolic ratio test \cite{bc_root}. Let $R=R_1e_1+R_2e_2$ denote the hyperbolic radius of convergent of the power series
\begin{align*}
    Z^{-\mathcal{V}} \mathcal{J}_\mathcal{V}(Z)=\sum_{s=0}^{\infty}\frac{(-1)^s}{\Gamma_b{(\mathcal{V}+s+1)}s!}\left(\frac{Z}{2}\right)^{2s} =\sum_{s=0}^{\infty}a_sY^s=\sum_{s=0}^{\infty}a_{1s}y_1^se_1+\sum_{s=0}^{\infty}a_{2s}y_2^se_2,
\end{align*}
where $a_{1s}=\frac{(-1)^s}{\Gamma(\nu_1+s+1)s!}$, $a_{2s}=\frac{(-1)^s}{\Gamma(\nu_2+s+1)s!}$ and $Y=\frac{Z^2}{4}$. Then,
\begin{align}\label{eq:mi}
     R=R_1e_1+R_2e_2&=\lim_{s\to \infty} \sup\frac{|a_s|_h}{|a_{s+1}|_h}\nonumber\\
    &=\lim_{s\to \infty} \sup\frac{|a_{1s}|}{|a_{1s+1}|}e_1+\lim_{s\to \infty} \sup\frac{|a_{2s}|}{|a_{2s+1}|}e_2\nonumber\\
    &=\lim_{s\to \infty} \sup\left|(s+1)(\nu_1+s+1)\right | e_1+ \lim_{s\to \infty} \sup\left|(s+1)(\nu_2+s+1)\right | e_2.
\end{align}
From \eqref{eq:mi}, we have $R_1=\infty$ and $R_2=\infty$. Then by hyperbolic ratio test \cite{bc_root}, the series $\sum\limits_{s=0}^{\infty}a_sY^s$ is absolutely hyperbolic convergent for all $Y\in\mathbb{BC}$ and that implies $Z^{-\mathcal{V}} \mathcal{J}_{\mathcal{V},C}(Z)$ absolutely hyperbolic convergent in bicomplex space.
Therefore, the proof has been established.
\end{proof}
\begin{theo}
 Suppose that  $Z=z_1e_1+z_2e_2, \mathcal{V}=\nu_1e_1+\nu_2e_2=-\frac{l_1+l_2}{2} +j\frac{l_1-l_2}{2i}\in{\mathbb{BC}}$, where $e_1=\frac{1+k}{2},e_2=\frac{1-k}{2}$ and  $l_1,l_2\in \mathbb{N}\cup \{0\}$, then $\mathcal{J}_{\mathcal{V}}(Z)=(-1)^{-\mathcal{V}}\mathcal{J}_{-\mathcal{V}}(Z)$.
\end{theo}
\begin{proof}
    Using idempotent representation of $\mathcal{V}$ we obtain $\nu_1=-l_1$ and $\nu_2=-l_2$ where $l_1,l_2\in \mathbb{N}\cup \{0\}$.
    Since $l_1$ and $l_2$ are non negative integer, the first $p-1$ terms in the infinite series of $\mathcal{J}_{\mathcal{V}}(Z)$ vanish because the pole of the gamma function in the denominator  and we obtain
\begin{align}
     \mathcal{J}_{\mathcal{V}}(Z)&=\sum_{s=0}^{\infty}\frac{(-1)^s}{\Gamma_b{(\mathcal{V}+s+1)}s!}\left(\frac{Z}{2}\right)^{\mathcal{V}+2s}\nonumber\\
  &=\sum_{s=0}^{\infty}\frac{(-1)^s}{\Gamma{(s-l_1+1)s!}}\left(\frac{z_1}{2}\right)^{2s-l_1}e_1+\sum_{s=0}^{\infty}\frac{(-1)^s}{\Gamma{(s-l_2+1)s!}}\left(\frac{z_2}{2}\right)^{2s-l_2}e_2\nonumber\\
  &=\sum_{s=l_1}^{\infty}\frac{(-1)^s}{\Gamma{(s-l_1+1)s!}}\left(\frac{z_1}{2}\right)^{2s-l_1}e_1+\sum_{s=l_2}^{\infty}\frac{(-1)^s}{\Gamma{(s-l_2+1)s!}}\left(\frac{z_2}{2}\right)^{2s-l_2}e_2\nonumber\\
  &=\sum_{s=0}^{\infty}\frac{(-1)^{s+l_1}}{\Gamma{(s+1)}(s+l_1)!}\left(\frac{z_1}{2}\right)^{2s+l_1}e_1+\sum_{s=0}^{\infty}\frac{(-1)^{s+l_2}}{\Gamma{(s+1)}(s+l_2)!}\left(\frac{z_2}{2}\right)^{2s+l_2}e_2\nonumber\\
  &=\sum_{s=0}^{\infty}\frac{(-1)^{s+l_1}}{\Gamma(l_1+s+1)s!}\left(\frac{z_1}{2}\right)^{2s+l_1}e_1+\sum_{s=0}^{\infty}\frac{(-1)^{s+l_2}}{\Gamma{(l_2+s+1)}s!}\left(\frac{z_2}{2}\right)^{2s+l_2}e_2\nonumber\\
  &=(-1)^{-\mathcal{V}}\sum_{s=0}^{\infty}\frac{(-1)^s}{\Gamma_b{(-\mathcal{V}+s+1)}s!}\left(\frac{Z}{2}\right)^{-\mathcal{V}+2s}\nonumber\\
  &=(-1)^{-\mathcal{V}}\mathcal{J}_{-\mathcal{V}}(Z)\nonumber.
\end{align}
Hence, the proof is complete.
\end{proof}
\begin{theo}
   Let  $Z=z_1e_1+z_2e_2, \mathcal{V}=\nu_1e_1+\nu_2e_2=\alpha_1 +j\alpha_2, M=me_1+ne_2 \in{\mathbb{BC}}$, where $e_1=\frac{1+k}{2},e_2=\frac{1-k}{2}$ and $m,n\in\mathbb{Z}^+$, then the bicomplex Bessel function satisfies the following recurrence relations:
  \begin{description}
         \item[(i)] $Z\mathcal{J}_{\mathcal{V}}(Z)=2(\mathcal{V}+1)\mathcal{J}_{\mathcal{V}+1}(Z)-ZJ_{\mathcal{V}+2}(Z)$;
         \item[(ii)] $Z^2\mathcal{J}_{\mathcal{V}}(Z)=4\mathcal{J}_{\mathcal{V}+2}(Z)-4Z(\mathcal{V}+2)\mathcal{J}_{\mathcal{V}+3}(Z)+Z^2\mathcal{J}_{\mathcal{V}+4}(Z);$
         \item[(iii)] $\mathcal{J}_{\mathcal{V}+M}(Z)+\mathcal{J}_{\mathcal{V}+\overline{M}}(Z)=\mathcal{J}_{\mathcal{V}+m}(Z)+\mathcal{J}_{\mathcal{V}+n}(Z)$
      \end{description}
\end{theo}
\begin{proof}(i)
     Based on the series representation of the bicomplex Bessel function given in equation \eqref{eq:gh}, we have
     \begin{align}\label{eq:sl}
         Z\mathcal{J}_{\mathcal{V}}(Z)&=2\sum_{s=0}^{\infty}\frac{(-1)^s}{\Gamma_b{(\mathcal{V}+s+1)}s!}\left(\frac{Z}{2}\right)^{\mathcal{V}+2s+1}\nonumber\\
    &=2\sum_{s=0}^{\infty}\frac{(-1)^s(\mathcal{V}+s+1)}{\Gamma_b{(\mathcal{V}+s+2)}s!}\left(\frac{Z}{2}\right)^{\mathcal{V}+2s+1}\nonumber\\
         &=2(\mathcal{V}+1)\sum_{s=0}^{\infty}\frac{(-1)^s}{\Gamma_b{(\mathcal{V}+s+2)}s!}\left(\frac{Z}{2}\right)^{\mathcal{V}+2s+1}+2\sum_{s=0}^{\infty}\frac{(-1)^s}{\Gamma_b{(\mathcal{V}+s+2)}(s-1)!}\left(\frac{Z}{2}\right)^{\mathcal{V}+2s+1}\nonumber\\
         &=2(\mathcal{V}+1)\mathcal{J}_{\mathcal{V}+1}(Z)+\sum_{s=0}^{\infty}\frac{(-1)^{s+1}}{\Gamma_b{(\mathcal{V}+s+3)}s!}\left(\frac{Z}{2}\right)^{\mathcal{V}+2s+2}\nonumber\\
         &=2(\mathcal{V}+1)\mathcal{J}_{\mathcal{V}+1}(Z)-Z\mathcal{J}_{\mathcal{V}+2}(Z).
     \end{align}
    Hence the proof of (i) is complete.\\
    (ii) Using \eqref{eq:gh}, we get
        \begin{align*}
           &Z^2\mathcal{J}_{\mathcal{V}}(Z)\\&=Z^2\sum_{s=0}^{\infty}\frac{(-1)^s}{\Gamma_b{(\mathcal{V}+s+1)}s!}\left(\frac{Z}{2}\right)^{\mathcal{V}+2s}\\
           &=Z^2\sum_{s=0}^{\infty}\frac{(-1)^s(\mathcal{V}+s+1)(\mathcal{V}+s+2)}{\Gamma_b{(\mathcal{V}+s+3)}s!}\left(\frac{Z}{2}\right)^{\mathcal{V}+2s}\\
           &=4\sum_{s=0}^{\infty}\frac{(-1)^s(\mathcal{V}+1)(\mathcal{V}+2)}{\Gamma_b{(\mathcal{V}+s+3)}s!}\left(\frac{Z}{2}\right)^{\mathcal{V}+2s+2}+8\sum_{s=0}^{\infty}\frac{(-1)^s(\mathcal{V}+2)}{\Gamma_b{(\mathcal{V}+s+3)}(s-1)!}\left(\frac{Z}{2}\right)^{\mathcal{V}+2s+2}\\&+4\sum_{s=0}^{\infty}\frac{(-1)^s}{\Gamma_b{(\mathcal{V}+s+3)}(s-2)!}\left(\frac{Z}{2}\right)^{\mathcal{V}+2s+2}\\
           &=4\mathcal{J}_{\mathcal{V}+2}(Z)+8\sum_{s=1}^{\infty}\frac{(-1)^s(\mathcal{V}+2)}{\Gamma_b{(\mathcal{V}+s+3)}(s-1)!}\left(\frac{Z}{2}\right)^{\mathcal{V}+2s+2}+4\sum_{s=2}^{\infty}\frac{(-1)^s}{\Gamma_b{(\mathcal{V}+s+3)}(s-2)!}\left(\frac{Z}{2}\right)^{\mathcal{V}+2s+2}\\
           &=4\mathcal{J}_{\mathcal{V}+2}(Z)+8\sum_{s=0}^{\infty}\frac{(-1)^{s+1}(\mathcal{V}+2)}{\Gamma_b{(\mathcal{V}+s+4)}s!}\left(\frac{Z}{2}\right)^{\mathcal{V}+2s+4}+4\sum_{s=0}^{\infty}\frac{(-1)^{s+2}}{\Gamma_b{(\mathcal{V}+s+5)}s!}\left(\frac{Z}{2}\right)^{\mathcal{V}+2s+6}\\
           &=4\mathcal{J}_{\mathcal{V}+2}(Z)-4Z(\mathcal{V}+2)\mathcal{J}_{\mathcal{V}+3}(Z)+Z^2\mathcal{J}_{\mathcal{V}+4}(Z).
        \end{align*}
        (iii) Since, $M=me_1+ne_2\in\mathbb{BC}$ where $m,n\in\mathbb{Z}^+$, then $\overline{M}=ne_1+me_2$. Using \eqref{eq:gh} and \eqref{eq:g}, we obtain
        \begin{align}\label{eq:6k}
            \mathcal{J}_{\mathcal{V}+M}(Z)&=\sum_{s=0}^{\infty}\frac{(-1)^s}{\Gamma_b{(\mathcal{V}+M+s+1)}s!}\left(\frac{Z}{2}\right)^{\mathcal{V}+2s}\nonumber\\
            &=\sum_{s=0}^{\infty}\frac{(-1)^s}{\Gamma{(\nu_1+m+s+1)}s!}\left(\frac{z_1}{2}\right)^{\nu_1+2s}e_1+\sum_{s=0}^{\infty}\frac{(-1)^s}{\Gamma{(\nu_2+n+s+1)}s!}\left(\frac{z_2}{2}\right)^{\nu_2+2s}e_2.
        \end{align}
        Similarly, we get
        \begin{align}\label{eq:2t}
            \mathcal{J}_{\mathcal{V}+\overline{M}}(Z)=\sum_{s=0}^{\infty}\frac{(-1)^s}{\Gamma{(\nu_1+n+s+1)}s!}\left(\frac{z_1}{2}\right)^{\nu_1+2s}e_1+\sum_{s=0}^{\infty}\frac{(-1)^s}{\Gamma{(\nu_2+m+s+1)}s!}\left(\frac{z_2}{2}\right)^{\nu_2+2s}e_2.
        \end{align}
        Combining \eqref{eq:6k} and \eqref{eq:2t}, we have
        \begin{align*}
            &\mathcal{J}_{\mathcal{V}+M}(Z)+\mathcal{J}_{\mathcal{V}+\overline{M}}(Z)\\&=\left[\sum_{s=0}^{\infty}\frac{(-1)^s}{\Gamma{(\nu_1+m+s+1)}s!}\left(\frac{z_1}{2}\right)^{\nu_1+2s}e_1+\sum_{r=0}^{\infty}\frac{(-1)^s}{\Gamma{(\nu_2+m+s+1)}s!}\left(\frac{z_2}{2}\right)^{\nu_2+2s}e_2\right]\\
            &+\left[\sum_{s=0}^{\infty}\frac{(-1)^s}{\Gamma{(\nu_1+n+s+1)}s!}\left(\frac{z_1}{2}\right)^{\nu_1+2s}e_1+\sum_{s=0}^{\infty}\frac{(-1)^s}{\Gamma{(\nu_2+n+s+1)}s!}\left(\frac{z_2}{2}\right)^{\nu_2+2s}e_2\right]\\
            &=\mathcal{J}_{\mathcal{V}+m}(Z)+\mathcal{J}_{\mathcal{V}+n}(Z).
        \end{align*}
        Hence the proof is completed.
 \end{proof}
     \section{Integral representations of the bicomplex Bessel function}
 This section focuses on deriving various forms of integral representations for the bicomplex Bessel function.
 \begin{theo}\label{th:ad}
    Let $Z=z_1e_1+z_2e_2, \mathcal{V}=\nu_1e_1+\nu_2e_2=\alpha_1 +j\alpha_2 \in{\mathbb{BC}}$, where $e_1=\frac{1+k}{2}$  and  $e_2=\frac{1-k}{2}$ with satisfies the condition $\Re(\alpha_1)>|\Im(\alpha_2)|$, then bicomplex Bessel function is represented integrally by
    \begin{align}\label{eq:a0}
        \mathcal{J}_\mathcal{V}(Z)=\frac{1}{\Gamma_b(\mathcal{V})}\left[\sum_{s=0}^{\infty}\frac{(-1)^s}{(s!)^2}\left(\frac{Z}{2}\right)^{\mathcal{V}+2s}\int_{D}t^{s}(1-t)^{\mathcal{V}-1}dt\right],
    \end{align}
    where $D=(D_1(t_1),D_2(t_2))$ be the curve in $ \mathbb{BC}$ and $t=t_1e_1+t_2e_2\in\mathbb{BC}$ with $0\leq t_1,t_2\leq 1$.
 \end{theo}
 \begin{proof}
    Based on the idempotent form \eqref{eq:fa} of the bicomplex Bessel function, we derive the following
    \begin{align*}
   &\mathcal{J}_\mathcal{V}(Z)\\&=J_{\nu_1}(z_1)e_1+J_{\nu_2}(z_2)e_2\\&=\sum_{s=0}^{\infty}\frac{(-1)^s}{\Gamma{(\nu_1+s+1)s!}}\left(\frac{z_1}{2}\right)^{\nu_1+2s}e_1+\sum_{s=0}^{\infty}\frac{(-1)^s}{\Gamma{(\nu_2+s+1)s!}}\left(\frac{z_2}{2}\right)^{\nu_2+2s}e_2\nonumber\\
    &=\left[\sum_{s=0}^{\infty}\frac{\beta(\nu_1,s+1)(-1)^s}{\Gamma(\nu_1)\Gamma(s+1)s!}\left(\frac{z_1}{2}\right)^{\nu_1+2s}\right]e_1+\left[\sum_{s=0}^{\infty}\frac{\beta(\nu_2,s+1)(-1)^s}{\Gamma(\nu_2)\Gamma(s+1)s!}\left(\frac{z_2}{2}\right)^{\nu_2+2s}\right]e_2\\&=\left[\sum_{s=0}^{\infty}\frac{(-1)^s}{\Gamma(\nu_1)\Gamma(s+1)s!}\left(\frac{z_1}{2}\right)^{\nu_1+2s}\int_{0}^{1}t_1^{s}(1-t_1)^{\nu_1-1}dt_1\right]e_1\\&\quad+\left[\sum_{s=0}^{\infty}\frac{(-1)^s}{\Gamma(\nu_2)\Gamma(s+1)s!}\left(\frac{z_2}{2}\right)^{\nu_2+2s}\int_{0}^{1}t_2^{s}(1-t_2)^{\nu_2-1}dt_2\right]e_2\\
   &=\frac{1}{\Gamma(\nu_1)}\left[\sum_{s=0}^{\infty}\frac{(-1)^s}{(s!)^2}\left(\frac{z_1}{2}\right)^{\nu_1+2s}\int_{D_1}t_1^{s}(1-t_1)^{\nu_1-1}dt_1\right]e_1\\&\quad+\frac{1}{\Gamma(\nu_1)}\left[\sum_{s=0}^{\infty}\frac{(-1)^s}{(s!)^2}\left(\frac{z_2}{2}\right)^{\nu_2+2s}\int_{D_2}t_1^{s}(1-t_2)^{\nu_2-1}dt_2\right]e_2\\
   &=\frac{1}{\Gamma_b(\mathcal{V})}\left[\sum_{s=0}^{\infty}\frac{(-1)^sZ^{\mathcal{V}+2s}}{4^s(s!)^2}\int_{D}t^{s}(1-t)^{\mathcal{V}-1}dt\right].
\end{align*}
Thus, the proof is finalized.
 \end{proof}
\begin{theo}
    Assume that $Z=z_1e_1+z_2e_2, \mathcal{V}=\nu_1e_1+\nu_2e_2=\alpha_1 +j\alpha_2 \in{\mathbb{BC}}$, where $e_1=\frac{1+k}{2}$  and  $e_2=\frac{1-k}{2}$ with satisfies $\Re(\alpha_1)+\frac{1}{2}>|\Im(\alpha_2)|$, then bicomplex Bessel function can be written as
    \begin{align*}
        \mathcal{J}_\mathcal{V}(Z)=\frac{Z^{\mathcal{V}}}{2^{\mathcal{V}-1}\sqrt{\pi}\Gamma_b\left(\mathcal{V}+\frac{1}{2}\right)}\sum_{s=0}^{\infty}\frac{(-1)^s}{\Gamma(2s+1)}\int_{D}(\cos t)^{2\mathcal{V}}(Z\sin t)^{2s}dt,
    \end{align*}
    where $D=(D_1(t_1),D_2(t_2))$ be the curve in $ \mathbb{BC}$ and $t=t_1e_1+t_2e_2\in\mathbb{BC}$ with $0\leq t_1,t_2\leq \frac{\pi}{2}$.
 \end{theo}
 \begin{proof}
     Applying the duplication formula for the gamma function \cite{sp_function}
     \begin{align*}
         \sqrt{\pi}\Gamma(2k+1)=2^{2k}\Gamma\left(k+\frac{1}{2}\right)\Gamma(k+1)
     \end{align*}
     in \eqref{eq:fa}, we obtain
      \begin{align}
          &\mathcal{J}_\mathcal{V}(Z)\\&= \sum_{s=0}^{\infty}\frac{(-1)^s}{\Gamma{(\nu_1+s+1)s!}}\left(\frac{z_1}{2}\right)^{\nu_1+2s}e_1+\sum_{s=0}^{\infty}\frac{(-1)^s}{\Gamma{(\nu_2+s+1)s!}}\left(\frac{z_2}{2}\right)^{\nu_2+2s}e_2\nonumber\\
          &=\sum_{s=0}^{\infty}\frac{(-1)^s\Gamma\left(s+\frac{1}{2}\right)}{\Gamma(\nu_1+s+1)s!\Gamma\left(s+\frac{1}{2}\right)}\left(\frac{z_1}{2}\right)^{\nu_1+2s}e_1+\sum_{s=0}^{\infty}\frac{(-1)^s\Gamma\left(s+\frac{1}{2}\right)}{\Gamma{(\nu_2+s+1)s!\Gamma\left(s+\frac{1}{2}\right)}}\left(\frac{z_2}{2}\right)^{\nu_2+2s}e_2\nonumber\\
          &=\sum_{s=0}^{\infty}\frac{(-1)^s2^{2s}\Gamma\left(s+\frac{1}{2}\right)}{\Gamma(\nu_1+s+1)\sqrt{\pi}\Gamma(2s+1)}\left(\frac{z_1}{2}\right)^{\nu_1+2s}e_1+\sum_{s=0}^{\infty}\frac{(-1)^s2^{2s}\Gamma\left(s+\frac{1}{2}\right)}{\Gamma(\nu_2+s+1)\sqrt{\pi}\Gamma(2s+1)}\left(\frac{z_2}{2}\right)^{\nu_2+2s}e_2\nonumber\\
          &=\frac{1}{2^{\nu_1}\sqrt{\pi}\Gamma\left(\nu_1+\frac{1}{2}\right)}\sum_{s=0}^{\infty}\frac{(-1)^sz_1^{\nu_1+2s}}{\Gamma(2s+1)}B(\nu_1+\frac{1}{2},s+\frac{1}{2})e_1\nonumber\\
          &\;\;\;+\frac{1}{2^{\nu_2}\sqrt{\pi}\Gamma\left(\nu_2+\frac{1}{2}\right)}\sum_{s=0}^{\infty}\frac{(-1)^sz_2^{\nu_2+2s}}{\Gamma(2s+1)}B(\nu_2+\frac{1}{2},s+\frac{1}{2})e_2\nonumber\\
          &=\frac{1}{2^{\nu_1-1}\sqrt{\pi}\Gamma\left(\nu_1+\frac{1}{2}\right)}\sum_{s=0}^{\infty}\frac{(-1)^sz_1^{\nu_1+2s}}{\Gamma(2s+1)}\int_{0}^{\frac{\pi}{2}}(\cos t_1)^{2\nu_1}(\sin t_1)^{2s}dt_1e_1\nonumber\\
          &\;\;\;+\frac{1}{2^{\nu_2-1}\sqrt{\pi}\Gamma\left(\nu_2+\frac{1}{2}\right)}\sum_{s=0}^{\infty}\frac{(-1)^sz_2^{\nu_2+2s}}{\Gamma(2s+1)}\int_{0}^{\frac{\pi}{2}}(\cos t_2)^{2\nu_2}(\sin t_2)^{2s}dt_2e_2\nonumber\\
          &=\left[\frac{z_1^{\nu_1}}{2^{\nu_1-1}\sqrt{\pi}\Gamma\left(\nu_1+\frac{1}{2}\right)}e_1+\frac{z_2^{\nu_2}}{2^{\nu_2-1}\sqrt{\pi}\Gamma\left(\nu_2+\frac{1}{2}\right)}e_2\right]\nonumber\\ &\;\;\times\sum_{s=0}^{\infty}\frac{(-1)^s}{\Gamma(2s+1)}\int_{(D_1,D_2)}[(\cos t_1)^{2\nu_1}(z_1\sin t_2)^{2s}e_1+(\cos t_2)^{2\nu_2}(z_2\sin t_2)^{2s}e_2][dt_1e_1+dt_2e_2]\nonumber\\
          &=\frac{Z^{\mathcal{V}}}{2^{\mathcal{V}-1}\sqrt{\pi}\Gamma_b\left(\mathcal{V}+\frac{1}{2}\right)}\sum_{s=0}^{\infty}\frac{(-1)^s}{\Gamma_b(2s+1)}\int_{D}(\cos t)^{2\mathcal{V}}(Z\sin t)^{2s}dt.\nonumber
      \end{align}
      Hence, the proof is completed.
 \end{proof}
 \begin{theo}\label{th:mn}
     Assume that $Z=z_1e_1+z_2e_2, \mathcal{V}=\nu_1e_1+\nu_2e_2=\alpha_1 +j\alpha_2, \delta=\delta_1e_1+\delta_2e_2=\beta_1+j\beta_2, C=c_1e_1+c_2e_2 \in{\mathbb{BC}}$, where $e_1=\frac{1+k}{2}, e_2=\frac{1-k}{2}$ and satisfies the conditions $\Re(\alpha_1)>|\Im(\alpha_2)|,\Re(\beta_1)+1>|\Im(\beta_2)|$, the integral representation of bicomplex Bessel function is given by
     \begin{align*}
         \mathcal{J}_{\mathcal{V}+\delta}(Z)=\frac{1}{\Gamma(\mathcal{V})\Gamma(\delta)}\left(\frac{Z}{2}\right)^{\mathcal{V}+\delta}\sum_{s=0}^{\infty}\frac{Z^{2s}}{4^s(s!)^2}\int_A\int_Du^{\mathcal{V}-1}(1-u)^{\delta+s}v^{\delta-1}(v-1)^sdudv,
     \end{align*}
     where $A=(A_1(u_1),A_2(u_2))$ and $D=(D_1(v_1),D_2(v_2))$ are two curves in $\mathbb{BC}$ with $0\leq u_1,u_2,v_1,v_2\leq1$.
 \end{theo}
 \begin{proof}
   Using \eqref{eq:gh}, we obtain
   \begin{align}\label{eq:zm}
        &\mathcal{J}_{\mathcal{V}+\delta}(Z)\nonumber\\&=\sum_{s=0}^{\infty}\frac{(-1)^s}{\Gamma_b{(\mathcal{V}+\delta+s+1)}s!}\left(\frac{Z}{2}\right)^{\mathcal{V}+\delta+2s}\nonumber\\
        &=\sum_{s=0}^{\infty}\frac{(-1)^s}{\Gamma{(\nu_1+\delta_1+s+1)}s!}\left(\frac{z_1}{2}\right)^{\nu_1+\delta_1+2s}e_1+\sum_{s=0}^{\infty}\frac{(-1)^s}{\Gamma{(\nu_2+\delta_2+s+1)}s!}\left(\frac{z_2}{2}\right)^{\nu_2+\delta_2+2s}e_2\nonumber\\
        &=\frac{1}{\Gamma(\nu_1)\Gamma(\delta_1)}\sum_{s=0}^{\infty}\frac{(-1)^s\Gamma(\nu_1)\Gamma(\delta_1+s+1)}{\Gamma{(\nu_1+\delta_1+s+1)}}\frac{\Gamma(\delta_1)\Gamma(s+1)}{\Gamma(\delta_1+s+1)(s!)^2}\left(\frac{z_1}{2}\right)^{\nu_1+\delta_1+2s}e_1\nonumber\\
        &\quad+\frac{1}{\Gamma(\nu_2)\Gamma(\delta_2)}\sum_{s=0}^{\infty}\frac{(-1)^s\Gamma(\nu_2)\Gamma(\delta_2+s+1)}{\Gamma{(\nu_2+\delta_2+s+1)}}\frac{\Gamma(\delta_2)\Gamma(s+1)}{\Gamma(\delta_2+s+1)(s!)^2}\left(\frac{z_2}{2}\right)^{\nu_2+\delta_2+2s}e_2\nonumber\\
        &=\frac{1}{\Gamma(\nu_1)\Gamma(\delta_1)}\sum_{s=0}^{\infty}\frac{(-1)^s}{(s!)^2}B(\nu_1,\delta_1+s+1)B(\delta_1,s+1)\left(\frac{z_1}{2}\right)^{\nu_1+\delta_1+2s}e_1\nonumber\\
        &\quad+\frac{1}{\Gamma(\nu_2)\Gamma(\delta_2)}\sum_{s=0}^{\infty}\frac{(-1)^s}{(s!)^2}B(\nu_2,\delta_2+s+1)B(\delta_2,s+1)\left(\frac{z_2}{2}\right)^{\nu_2+\delta_2+2s}e_2\nonumber\\
       & =\frac{1}{\Gamma(\nu_1)\Gamma(\delta_1)}\sum_{s=0}^{\infty}\frac{1}{(s!)^2}\int_0^1u_1^{\nu_1-1}(1-u_1)^{\delta_1+s}du_1\int_0^1v_1^{\delta_1-1}(v_1-1)^sdv_1\left(\frac{z_1}{2}\right)^{\nu_1+\delta_1+2s}e_1\nonumber\\
       &\quad+\frac{1}{\Gamma(\nu_2)\Gamma(\delta_2)}\sum_{s=0}^{\infty}\frac{1}{(s!)^2}\int_0^1u_2^{\nu_2-1}(1-u_2)^{\delta_2+s}du_2\int_0^1v_2^{\delta_2-1}(v_2-1)^sdv_2\left(\frac{z_2}{2}\right)^{\nu_2+\delta_2+2s}e_2.
   \end{align}
Let us consider two curves $A$ and $D$ in $\mathbb{BC}$, whose parametric representation are $A=(A_1(u_1),A_2(u_2))$ and $D=(D_1(v_1),D_2(v_2))$ respectively and $u=u_1e_1+u_2e_2, v=v_1e_1+v_2e_2\in{\mathbb{BC}}$ with satisfies the condition $0\leq u_1,u_2,v_1,v_2\leq1$. Now from equation \eqref{eq:zm}, we obtain
\begin{align*}
    & \mathcal{J}_{\mathcal{V}+\delta}(Z)\\&= \left[\frac{1}{\Gamma(\nu_1)\Gamma(\delta_1)}\left(\frac{z_1}{2}\right)^{\nu_1+\delta_1}e_1+\frac{1}{\Gamma(\nu_2)\Gamma(\delta_2)}\left(\frac{z_2}{2}\right)^{\nu_2+\delta_2}e_2\right]\\&\quad\times\sum_{s=0}^{\infty}\frac{1}{(s!)^2}\int_A(u_1e_1+u_2e_2)^{\nu_1e_1+\nu_2e_2-1}(1-u_1e_1-u_2e_2)^{\delta_1e_1+\delta_2e_2+s}d(u_1e_1+u_2e_2)\\
 &\hspace{40pt}\int_D(v_1e_1+v_2e_2)^{\delta_1e_1+\delta_2e_2-1}(v_1e_1+v_2e_2-1)^sd(v_1e_1+v_2e_2)\left[\left(\frac{z_1}{2}\right)^{2s}e_1+\left(\frac{z_2}{2}\right)^{2s}e_2\right]\\
     &=\frac{1}{\Gamma(\mathcal{V})\Gamma(\delta)}\left(\frac{Z}{2}\right)^{\mathcal{V}+\delta}\sum_{s=0}^{\infty}\frac{Z^{2s}}{4^s(s!)^2}\int_A\int_Du^{\mathcal{V}-1}(1-u)^{\delta+s}v^{\delta-1}(v-1)^sdudv,
\end{align*}
 which completes the proof of the theorem.
 \end{proof}

  \begin{theo}\label{th:w9}
      Suppose that  $Z=z_1e_1+z_2e_2, \mathcal{V}=\nu_1e_1+\nu_2e_2=\alpha_1 +j\alpha_2 \in{\mathbb{BC}}$, where $e_1=\frac{1+k}{2}$ and $ e_2=\frac{1-k}{2}$  with satisfies the condition $\Re(\alpha_1)+\frac{1}{2}>|\Im(\alpha_2)|$, then  bicomplex Bessel function can be expresses as
     \begin{align*}
         \mathcal{J}_{\mathcal{V}}(Z)=\frac{2^{\mathcal{V}}}{\sqrt{\pi}}\sum_{s=0}^{\infty}\frac{(-1)^sZ^{\mathcal{V}+2s}}{s!\Gamma_b(2\mathcal{V}+2s+1)}\int_{\gamma}e^{-t}t^{\mathcal{V}+s-\frac{1}{2}}dt,
     \end{align*}
    where $t=t_1e_1+t_2e_2\in\mathbb{BC}$ and $\gamma=(\gamma_1,\gamma_2)$ be a curve in $ \mathbb{BC}$, defined by the following parametric equation $\gamma(t)=(\gamma(t_1),\gamma(t_2))$ with $0<t_1,t_2<\infty$.
 \end{theo}
 \begin{proof}
 Using Legendre's duplication formula of Gamma function \cite{sp_function} and \eqref{eq:fa}, we get
\begin{align*}
   &\mathcal{J}_\mathcal{V}(Z)=\sum_{s=0}^{\infty}\frac{(-1)^s}{\Gamma{(\nu_1+s+1)s!}}\left(\frac{z_1}{2}\right)^{\nu_1+2s}e_1+\sum_{s=0}^{\infty}\frac{(-1)^s}{\Gamma{(\nu_2+s+1)s!}}\left(\frac{z_2}{2}\right)^{\nu_2+2s}e_2\\&=\sum_{s=0}^{\infty}\frac{\Gamma(\nu_1+s+\frac{1}{2})(-1)^s}{\Gamma(\nu_1+s+\frac{1}{2})\Gamma{(\nu_1+s+1)s!}}\left(\frac{z_1}{2}\right)^{\nu_1+2s}e_1+\sum_{s=0}^{\infty}\frac{\Gamma(\nu_2+s+\frac{1}{2})(-1)^s}{\Gamma(\nu_2+s+\frac{1}{2})\Gamma{(\nu_2+s+1)s!}}\left(\frac{z_2}{2}\right)^{\nu_2+2s}e_2\\&=\sum_{s=0}^{\infty}\frac{2^{2\nu_1+2s}\Gamma(\nu_1+s+\frac{1}{2})(-1)^s}{\sqrt{\pi}\Gamma(2\nu_1+2s+1)s!}\left(\frac{z_1}{2}\right)^{\nu_1+2s}e_1+\sum_{s=0}^{\infty}\frac{2^{2\nu_2+2s}\Gamma(\nu_2+s+\frac{1}{2})(-1)^s}{\sqrt{\pi}\Gamma(2\nu_2+2s+1)s!}\left(\frac{z_2}{2}\right)^{\nu_2+2s}e_2\\
   &=\frac{1}{\sqrt{\pi}}\left[\sum_{s=0}^{\infty}\frac{(-1)^s2^{\nu_1}z_1^{\nu_1+2s}}{\Gamma(2\nu_1+2s+1)s!}\int_{0}^{\infty}e^{-t_1}t_1^{\nu_1+s-\frac{1}{2}}dt_1e_1+\sum_{s=0}^{\infty}\frac{(-1)^s2^{\nu_2}z_2^{\nu_2+2s}}{\Gamma(2\nu_2+2s+1)s!}\int_{0}^{\infty}e^{-t_2}t_2^{\nu_2+s-\frac{1}{2}}dt_2e_2\right]\\
   &=\frac{1}{\sqrt{\pi}}\sum_{s=0}^{\infty}\frac{(-1)^s}{s!}\left[\frac{2^{\nu_1}z_1^{\nu_1+2s}}{\Gamma(2\nu_1+2s+1)}e_1+\frac{2^{\nu_2}z_2^{\nu_2+2s}}{\Gamma(2\nu_2+2s+1)}e_2\right]\\&\hspace{80pt}\times\left[\int_{\gamma_1}e^{-t_1}t_1^{\nu_1+s-\frac{1}{2}}dt_1e_1+\int_{\gamma_2}e^{-t_2}t_2^{\nu_2+s-\frac{1}{2}}dt_2e_2\right]\\
   &=\frac{2^{\mathcal{V}}}{\sqrt{\pi}}\sum_{s=0}^{\infty}\frac{(-1)^sZ^{\mathcal{V}+2s}}{s!\Gamma_b(2\mathcal{V}+2s+1)}\int_{\gamma}e^{-t}t^{\mathcal{V}+s-\frac{1}{2}}dt.
\end{align*}
 This concludes the proof of the theorem.
 \end{proof}
 \begin{theo}
     Suppose that $Z=z_1e_1+z_2e_2, W=w_1e_1+w_2e_2\in\mathbb{BC}$, where $e_1=\frac{1+k}{2}, e_2=\frac{1-k}{2}$ and a set of bicomplex valued function $G(Z,W)$ such that
     \begin{align}\label{eq:6r}
     G(Z,W)=\sum\limits_{n=-\infty}^{\infty}\mathcal{J}_n(Z)W^n,
     \end{align}
     then $G(Z,W)=\exp\left[\frac{1}{2}Z\left(W-\frac{1}{W}\right)\right]$.
 \end{theo}
 \begin{proof}
     First, substitute the series expansion of bicomplex Bessel function \eqref{eq:gh} in \eqref{eq:6r}, we get
     \begin{align}\label{eq:c9}
          &G(Z,W)\nonumber\\&=\sum\limits_{n=-\infty}^{\infty}\sum_{s=0}^{\infty}\frac{(-1)^s}{\Gamma_b{(n+s+1)}s!}\left(\frac{Z}{2}\right)^{n+2s}W^n\nonumber\\
          &=\sum\limits_{n=-\infty}^{\infty}\sum_{s=0}^{\infty}\frac{(-1)^s}{\Gamma{(n+s+1)}s!}\left(\frac{z_1}{2}\right)^{n+2s}w_1^ne_1+\sum\limits_{n=-\infty}^{\infty}\sum_{s=0}^{\infty}\frac{(-1)^s}{\Gamma{(n+s+1)}s!}\left(\frac{z_2}{2}\right)^{n+2s}w_2^ne_2\nonumber\\
          &=\sum_{s=0}^{\infty}\sum\limits_{n=-\infty}^{\infty}\frac{(-1)^s}{\Gamma{(n+s+1)}s!}\left(\frac{z_1}{2}\right)^{n+2s}w_1^ne_1+\sum_{s=0}^{\infty}\sum\limits_{n=-\infty}^{\infty}\frac{(-1)^s}{\Gamma{(n+s+1)}s!}\left(\frac{z_2}{2}\right)^{n+2s}w_2^ne_2.
   \end{align}
   We know $\Gamma(z)$ has a simple pole at each non positive integer, then for $n<-k$ we obtain $\frac{1}{\Gamma(n+k+1)}=0$. Therefore from the above equation \eqref{eq:c9}, we get
   \begin{align}
     &G(Z,W)\nonumber\\ &=\sum_{s=0}^{\infty}\sum\limits_{n=-2s}^{\infty}\frac{(-1)^s}{\Gamma{(n+s+1)}s!}\left(\frac{z_1}{2}\right)^{n+2s}w_1^ne_1+\sum_{s=0}^{\infty}\sum\limits_{n=-2s}^{\infty}\frac{(-1)^s}{\Gamma{(n+s+1)}s!}\left(\frac{z_2}{2}\right)^{n+2s}w_2^ne_2\nonumber\\
    &=\sum_{s=0}^{\infty}\sum\limits_{p=0}^{\infty}\frac{(-1)^s}{\Gamma{(p-s+1)}s!}\left(\frac{z_1}{2}\right)^{p}w_1^{p-2s}e_1+\sum_{s=0}^{\infty}\sum\limits_{q=0}^{\infty}\frac{(-1)^s}{\Gamma{(q-s+1)}s!}\left(\frac{z_2}{2}\right)^{q}w_2^{q-2s}e_2\nonumber
   \end{align}
   Again interchange the order of the summation and using the concept of the pole of $\Gamma(z)$, we have
   \begin{align*}
     &G(Z,W)\nonumber\\& =\sum_{p=0}^{\infty}\sum\limits_{s=0}^{p}\frac{(-1)^s}{\Gamma{(p-s+1)}s!}\left(\frac{z_1}{2}\right)^{p}w_1^{p-2s}e_1+\sum_{q=0}^{\infty}\sum\limits_{s=0}^{q}\frac{(-1)^s}{\Gamma{(q-s+1)}s!}\left(\frac{z_2}{2}\right)^{q}w_2^{q-2s}e_2\\
     &=\sum_{p=0}^{\infty}\left[\sum\limits_{s=0}^{p}\frac{w_1^{p-s}(-w_1^{-1})^s}{\Gamma{(p-s+1)}s!}\right]\left(\frac{z_1}{2}\right)^{p}e_1+\sum_{q=0}^{\infty}\left[\sum\limits_{s=0}^{q}\frac{w_2^{q-s}(-w_2^{-1})^s}{\Gamma{(q-s+1)}s!}\right]\left(\frac{z_2}{2}\right)^{q}e_2\\
     &=\sum_{p=0}^{\infty}\left(w_1-\frac{1}{w_1}\right)^p\frac{1}{p!}\left(\frac{z_1}{2}\right)^{p}e_1+\sum_{q=0}^{\infty}\left(w_2-\frac{1}{w_2}\right)^q\frac{1}{q!}\left(\frac{z_2}{2}\right)^{q}e_2\\
     &=\exp\left[\frac{z_1}{2}\left(w_1-\frac{1}{w_1}\right)\right]e_1+\exp\left[\frac{z_2}{2}\left(w_2-\frac{1}{w_2}\right)\right]e_2\\
     &=\exp\left[\frac{Z}{2}\left(W-\frac{1}{W}\right)\right].
      \end{align*}
      The proof is now complete.
 \end{proof}
\begin{corollary}
    From equation \eqref{eq:6r}, it is observe that $\mathcal{J}_n(Z)$ is the coefficients of the bicomplex Laurent series expansion of $\exp\left[\frac{1}{2}Z\left(W-\frac{1}{W}\right)\right]$ about $W=0$. According to the bicomplex Laurent theorem \cite{bc_laurent}, we obtain
    \begin{align*}
        \mathcal{J}_n(Z)=\frac{e_1}{2\pi i}\int_{C_1}W^{-n-1}\exp\left[\frac{1}{2}Z\left(W-\frac{1}{W}\right)\right]dW+\frac{e_2}{2\pi i}\int_{C_2}W^{-n-1}\exp\left[\frac{1}{2}Z\left(W-\frac{1}{W}\right)\right] dW,
    \end{align*}
    where $C_1$ and $C_2$ are contour encloses the singularity at the origin in $\mathcal{S}_1$ and $\mathcal{S}_2$ respectively.
\end{corollary}
 \section{Differential relations and differential equation of the bicomplex Bessel function}
 In this section, we develop differential equations that satisfy the bicomplex Bessel function and deduce differential relations.
 \begin{theo}
    Let  $Z=z_1e_1+z_2e_2, \mathcal{V}=\nu_1e_1+\nu_2e_2=\alpha_1 +j\alpha_2 \in{\mathbb{BC}}$ where $e_1=\frac{1+k}{2}, e_2=\frac{1-k}{2}$, then
    \begin{description}
    \item[(i)]
    $Z\mathcal{J}'_\mathcal{V}(Z)+\mathcal{V}\mathcal{J}_\mathcal{V}(Z)=Z\mathcal{J}_{\mathcal{V}-1}(Z)$
    \item [(ii)]
    $Z\mathcal{J}'_\mathcal{V}(Z)+Z\mathcal{J}_{\mathcal{V}+1}(Z)=\mathcal{V}\mathcal{J}_{\mathcal{V}}(Z)$
    \end{description}
\end{theo}
\begin{proof}
   (i) Applying the derivative formula for bicomplex functions along with equation \eqref{eq:gh}, we have
\begin{align*}
    &\frac{d}{dZ} \left( Z^{\mathcal{V}}\mathcal{J}_\mathcal{V}(Z)\right)\\
     &=\frac{d}{dz_1}\left(
  \sum_{s=0}^{\infty}\frac{(-1)^sz_1^{{2\nu_1+2s}}}{\Gamma{(\nu_1+s+1)}s!2^{{\nu_1+2s}}}\right)e_1+\frac{d}{dz_1}\left(
  \sum_{s=0}^{\infty}\frac{(-1)^sz_2^{{2\nu_2+2s}}}{\Gamma{(\nu_2+s+1)}s!2^{{\nu_2+2s}}}\right) e_2\nonumber\\
     &=\left(
  \sum_{s=0}^{\infty}\frac{(-1)^s(2\nu_1+2s)z_1^{{2\nu_1+2s-1}}}{\Gamma{(\nu_1+s+1)}s!2^{{\nu_1+2s}}}\right)e_1+\left(
  \sum_{s=0}^{\infty}\frac{(-1)^s(2\nu_2+2s)z_2^{{2\nu_2+2s-1}}}{\Gamma{(\nu_2+s+1)}s!2^{{\nu_2+2s}}}\right) e_2\nonumber\\
     &=\left(z_1^{\nu_1}
  \sum_{s=0}^{\infty}\frac{(-1)^s}{\Gamma{(\nu_1+s)}s!}\left(\frac{z_1}{2}\right)^{\nu_1+2s-1}\right)e_1+\left(z_2^{\nu_2}
  \sum_{s=0}^{\infty}\frac{(-1)^s}{\Gamma{(\nu_2+s)}s!}\left(\frac{z_2}{2}\right)^{\nu_2+2s-1}\right) e_2\\
    &=\left(z_1^{\nu_1}J_{\nu_1-1}\right)e_1+\left(z_2^{\nu_2}J_{\nu_2-1}\right)e_2\\
    &=Z^{\mathcal{V}}\mathcal{J}_{\mathcal{V}-1}(Z).
    \end{align*}
    Again,
    \begin{align*}
        \frac{d}{dZ} \left( Z^{\mathcal{V}}\mathcal{J}_\mathcal{V}(Z)\right)&=\frac{d}{dz_1}(z_1^{\nu_1}J_{\nu_1}(z_1))e_1+\frac{d}{dz_2}(z_2^{\nu_2}J_{\nu_2}(z_2))e_2\\
        &=(z_1^{\nu_1}J'_{\nu_1}(z_1)+\nu_1z_1^{\nu_1-1}J_{\nu_1}(z_1))e_1+(z_2^{\nu_2}J'_{\nu_2}(z_2)+\nu_2z_2^{\nu_2-1}J_{\nu_2}(z_2))e_2\\
        &=Z^{\mathcal{V}}\mathcal{J}'_\mathcal{V}(Z)+\mathcal{V}Z^{\mathcal{V}-1}\mathcal{J}_\mathcal{V}(Z)
    \end{align*}
    Hence, the proof of (i) is complete.\\
    (ii) Multiplying both side by $Z^{-\mathcal{V}}$ in \eqref{eq:gh} and differentiating, we get
    \begin{align*}
    &\frac{d}{dZ} \left( Z^{-\mathcal{V}}\mathcal{J}_\mathcal{V}(Z)\right)\\
     &=\frac{d}{dz_1}\left(
  \sum_{s=0}^{\infty}\frac{(-1)^sz_1^{{2s}}}{\Gamma{(\nu_1+s+1)}s!2^{{\nu_1+2s}}}\right)e_1+\frac{d}{dz_1}\left(
  \sum_{s=0}^{\infty}\frac{(-1)^sz_2^{{2s}}}{\Gamma{(\nu_2+s+1)}s!2^{{\nu_2+2s}}}\right) e_2\nonumber\\
     &=\left(
  \sum_{s=0}^{\infty}\frac{(-1)^s(2s)z_1^{{2s-1}}}{\Gamma{(\nu_1+s+1)}s!2^{{\nu_1+2s}}}\right)e_1+\left(
  \sum_{s=0}^{\infty}\frac{(-1)^s(2s)z_2^{{2s-1}}}{\Gamma{(\nu_2+s+1)}s!2^{{\nu_2+2s}}}\right) e_2\nonumber\\
     &=\left(
  \sum_{s=1}^{\infty}\frac{(-1)^s}{\Gamma{(\nu_1+s+1)}(s-1)!}\left(\frac{z_1}{2}\right)^{2s-1}\right)e_1+\left(
  \sum_{s=1}^{\infty}\frac{(-1)^s}{\Gamma{(\nu_2+s+1)}(s-1)!}\left(\frac{z_2}{2}\right)^{2s-1}\right) e_2\\
  &=\left[
  \sum_{s=0}^{\infty}\frac{(-1)^{(s+1)}}{\Gamma{(\nu_1+s+2)}s!}\left(\frac{z_1}{2}\right)^{2s+1}\right]e_1+\left[
  \sum_{s=0}^{\infty}\frac{(-1)^{s+1}}{\Gamma{(\nu_2+s+2)}s!}\left(\frac{z_2}{2}\right)^{2s+1}\right] e_2\\
    &=-\left(z_1^{-\nu_1}J_{\nu_1+1}\right)e_1-\left(z_2^{-\nu_2}J_{\nu_2+1}\right)e_2\\
    &=-Z^{-\mathcal{V}}\mathcal{J}_{\mathcal{V}+1}(Z).
    \end{align*}
     Again,
    \begin{align*}
        \frac{d}{dZ} \left( Z^{-\mathcal{V}}\mathcal{J}_\mathcal{V}(Z)\right)=Z^{-\mathcal{V}}\mathcal{J}'_\mathcal{V}(Z)-\mathcal{V}Z^{-\mathcal{V}-1}\mathcal{J}_\mathcal{V}(Z).
    \end{align*}
    Hence the proof of (ii) is complete.
\end{proof}

\begin{theo}
     Suppose that  $Z=z_1e_1+z_2e_2, \mathcal{V}=\nu_1e_1+\nu_2e_2 \in{\mathbb{BC}}$ with $Z\notin \mathbb{O}_2$. Bicomplex Bessel  function $ \mathcal{J}_\mathcal{V}(Z)$ satisfies the differential equation $$Z^2\frac{d^2V}{dZ^2}+Z\frac{dV}{dZ}+(Z^2-\mathcal{V}^2)V=0.$$
 \end{theo}
 \begin{proof}
 Another representation of bicomplex Bessel function is given by
  \begin{align}\label{eq:do}
     \mathcal{J}_\mathcal{V}(Z)&=
  \sum_{s=0}^{\infty}\frac{(-1)^s}{\Gamma_b{(\mathcal{V}+s+1)}s!}\left(\frac{Z}{2}\right)^{\mathcal{V}+2s}\nonumber\\
  &=\sum_{s=0}^{\infty}\frac{(-1)^s}{\Gamma{(\nu_1+s+1)}s!}\left(\frac{z_1}{2}\right)^{\nu_1+2s}e_1+\sum_{s=0}^{\infty}\frac{(-1)^s}{\Gamma{(\nu_2+s+1)}s!}\left(\frac{z_2}{2}\right)^{\nu_2+2s}e_2\nonumber\\
  &=\frac{1}{\Gamma(\nu_1+1)}\sum_{s=0}^{\infty}\frac{(-1)^s\Gamma(\nu_1+1)}{\Gamma{(\nu_1+s+1)}s!}\left(\frac{z_1}{2}\right)^{\nu_1+2s}e_1+\frac{1}{\Gamma(\nu_2+1)}\sum_{s=0}^{\infty}\frac{(-1)^s\Gamma(\nu_2+1)}{\Gamma{(\nu_2+s+1)}s!}\left(\frac{z_2}{2}\right)^{\nu_2+2s}e_2\nonumber\\
  &=\frac{z_1^{\nu_1}}{2^{\nu_1}\Gamma(\nu_1+1)}\sum_{s=0}^{\infty}\frac{1}{(\nu_1+1)_ss!}\left(-\frac{z_1^2}{4}\right)^{s}e_1+\frac{z_2^{\nu_2}}{2^{\nu_2}\Gamma(\nu_2+1)}\sum_{s=0}^{\infty}\frac{1}{(\nu_2+1)_ss!}\left(-\frac{z_2^2}{4}\right)^{s}e_2\nonumber\\
 &=\frac{z_1^{\nu_1}}{2^{\nu_1}\Gamma(\nu_1+1)}\times{}_{0} F_{1}\left[\begin{matrix} & -\; ;\\& \nu_1+1;&\end{matrix}-\frac{z_1^2}{4}\right]e_1+\frac{z_2^{\nu_2}}{2^{\nu_2}\Gamma(\nu_2+1)}\times{}_{0} F_{1}\left[\begin{matrix} & -\; ;\\& \nu_2+1;&\end{matrix}-\frac{z_2^2}{4}\right]e_2\nonumber\\
   &=\frac{Z^\mathcal{V}}{2^\mathcal{V}\Gamma_b(\mathcal{V}+1)}\times{}_{0} F_{1}\left[\begin{matrix} &-\; ;\\& \mathcal{V}+1;&\end{matrix}-\frac{Z^2}{4}\right].
  \end{align}
   Substituting $p=0$ and $q=1$ into equation \eqref{eq:in}, yields a specific solution $W={}_{0} F_{1}(-; b_1;Y)$ to the differential equation
   \begin{align}\label{eq:nk}
       Y\frac{d^2W}{dY^2}+b_1\frac{dW}{dY}-W=0.
   \end{align}
 Now we put  $b_1=\mathcal{V}+1$ and $Y=-\frac{Z^2}{4}$ in \eqref{eq:nk}, we obtain a differential equation
 \begin{align}\label{eq:a6}
     Z\frac{d^2W}{dZ^2}+(2\mathcal{V}+1)\frac{dW}{dZ}+ZW=0
 \end{align}
 whose one particular solution is $W={}_{0} F_{1}\left(-;\mathcal{V}+1;-\frac{Z^2}{4}\right)$. Again setting $W=Z^{-\mathcal{V}}V$ in \eqref{eq:a6}, we have
 \begin{align*}
     &Z\left(Z^{-\mathcal{V}}\frac{d^2V}{dZ^2}-2\mathcal{V}Z^{-(\mathcal{V}+1)}\frac{dV}{dZ}+\mathcal{V}(\mathcal{V}+1)Z^{-(\mathcal{V}+2)}V\right)\\&\hspace{90pt}+(2\mathcal{V}+1)\left(Z^\mathcal{-V}\frac{dV}{dZ}-\mathcal{V}Z^{-(\mathcal{V}+1)}V\right)
     +Z^{1\mathcal{-V}}V=0
 \end{align*}
 and it implies that
 \begin{align*}
     Z^2\frac{d^2V}{dZ^2}+Z\frac{dV}{dZ}+(Z^2-\mathcal{V}^2)V=0.
 \end{align*}
 This completes the proof.
 \end{proof}
\section{Asymptotic expansions and bicomplex holomorphicity of the bicomplex Bessel function}
Asymptotic expansion is a mathematical technique used to approximate functions, integrals, or solutions to differential equations. These expansions are particularly useful in analyzing the behavior of functions as a parameter approaches a certain limit, such as infinity or zero. Asymptotic expansions are used in quantum mechanics and electrodynamics, fluid Dynamics, statistical mechanics, number theory etc. In this section, we analyze the asymptotic behavior and discuss the bicomplex holomorphicity of bicomplex Bessel function $\mathcal{J}_\mathcal{V}(Z)$.
The asymptotic power series representation of a function $f(z)$ is given by \cite{sp_function}\;:
\begin{align*}f(z)\thicksim\sum_{k=0}^{\infty}a_kz^{-k},\; \mbox{as}\; |z|\rightarrow \infty,
\end{align*}
provided that
$$\lim_{|z|\rightarrow\infty}z^n\left[f(z)-\sum_{k=0}^{n}a_kz^{-k}\right]=0,$$ for each fixed $n$. We define asymptotic expansions of bicomplex valued function $g(Z)$ as
\begin{align*}
g(Z)\thicksim\sum_{k=0}^{\infty}b_kZ^{-k},\; |Z|_h\rightarrow \infty,
\end{align*}
if, for each fixed $n>0$, $$\lim_{|Z|_h\rightarrow\infty}Z^n\left[g(Z)-\sum_{k=0}^{n}b_kZ^{-k}\right]=0.$$
\begin{theo}
    If $Z=|Z|_h$, then asymptotic expansion of bicomplex Bessel function is given by
     \begin{align*}
        \mathcal{J}_\mathcal{V}(Z)&\thicksim\frac{\exp((1-i)Z)}{2^\mathcal{V}\Gamma_b(\mathcal{V}+1)}\frac{\Gamma_b(2\mathcal{V}+1)}{\left(\Gamma_b(\mathcal{V}+\frac{1}{2})\right)^2}\sum_{k=0}^{n-1}\frac{\left(-\mathcal{V}+\frac{1}{2}\right)_k}{k!Z^k}\Gamma_b(\mathcal{V}+k+\frac{1}{2}),\; \text{as}\;|Z|_h\rightarrow \infty.
     \end{align*}
 \end{theo}
 \begin{proof}
 We start with the representation \eqref{eq:do} of bicomplex Bessel function $\mathcal{J}_\mathcal{V}(Z)$\;:
  \begin{align}\label{eq:n4}
      \mathcal{J}_\mathcal{V}(Z)&=\frac{Z^\mathcal{V}}{2^\mathcal{V}\Gamma_b(\mathcal{V}+1)}\times{}_{0} F_{1}\left[\begin{matrix} &-\; ;\\& \mathcal{V}+1;&\end{matrix}-\frac{Z^2}{4}\right]\nonumber\\
      &=\frac{Z^\mathcal{V}}{2^\mathcal{V}\Gamma_b(\mathcal{V}+1)}\left[{}_{0} F_{1}\left[\begin{matrix} &-\; ;\\& \nu_1+1;&\end{matrix}-\frac{z_1^2}{4}\right]e_1+{}_{0} F_{1}\left[\begin{matrix} &-\; ;\\& \nu_2+1;&\end{matrix}-\frac{z_2^2}{4}\right]e_2\right]
  \end{align}
  Applying the relation of confluent hypergeometric function \cite{sp_function}: $$ e^{-z} {}_1F_{1}(c;2c;2z)={}_0F_1(-;c+\frac{1}{2};\frac{z^2}{4}), $$ and Lemma \ref{ls} in \eqref{eq:n4}, we get
  \begin{align}\label{eq:s9}
       \mathcal{J}_\mathcal{V}(Z)
       &=\frac{Z^\mathcal{V}}{2^\mathcal{V}\Gamma_b(\mathcal{V}+1)}\left[\exp(-iz_1){}_{1} F_{1}\left[\begin{matrix} &\nu_1+\frac{1}{2}\; ;\\& 2\nu_1+1;&\end{matrix}-2iz_1\right]e_1+\exp(-iz_2){}_{1} F_{1}\left[\begin{matrix} &\nu_2+\frac{1}{2}\; ;\\& 2\nu_2+1;&\end{matrix}2iz_2\right]e_2\right]\nonumber\\
       &=\frac{Z^\mathcal{V}\left[\exp(-iz_1)e_1+\exp(-iz_2)e_2\right]}{2^\mathcal{V}\Gamma_b(\mathcal{V}+1)}\left[{}_{1} F_{1}\left[\begin{matrix} &\nu_1+\frac{1}{2}\; ;\\& 2\nu_1+1;&\end{matrix}-2iz_1\right]e_1+{}_{1} F_{1}\left[\begin{matrix} &\nu_2+\frac{1}{2}\; ;\\& 2\nu_2+1;&\end{matrix}2iz_2\right]e_2\right]\nonumber\\
       &=\frac{Z^\mathcal{V}}{2^\mathcal{V}\Gamma_b(\mathcal{V}+1)}\exp(-iZ){}_{1} F_{1}\left[\begin{matrix} &\mathcal{V}+\frac{1}{2}\; ;\\& 2\mathcal{V}+1;&\end{matrix}-2iZ\right]\nonumber\\
       &=\frac{Z^\mathcal{V}\exp(-iZ)}{2^\mathcal{V}\Gamma_b(\mathcal{V}+1)}\frac{\Gamma_b(2\mathcal{V}+1)}{\left(\Gamma_b(\mathcal{V}+\frac{1}{2})\right)^2}\int_{C}t^{\mathcal{V}-\frac{1}{2}}(1-t)^{\mathcal{V}-\frac{1}{2}}\exp(Zt)dt,
  \end{align}
 where $t=t_1e_1+t_2e_2$ with $0<t_1,t_2<1$. Here $Z=|Z|_h$ and the integral of \eqref{eq:s9} can be express as
 \begin{align*}
     \mathcal{J}_\mathcal{V}(Z)=\frac{|Z|_h^\mathcal{V}\exp(-i|Z|_h)}{2^\mathcal{V}\Gamma_b(\mathcal{V}+1)}\frac{\Gamma_b(2\mathcal{V}+1)}{\left(\Gamma_b(\mathcal{V}+\frac{1}{2})\right)^2}\left[\int_{C_1}t^{\mathcal{V}-\frac{1}{2}}(1-t)^{\mathcal{V}-\frac{1}{2}}\exp(|Z|_ht)dt\right.\\\left.-\int_{C_2}t^{\mathcal{V}-\frac{1}{2}}(1-t)^{\mathcal{V}-\frac{1}{2}}\exp(|Z|_ht)dt\right],
 \end{align*}
 where $C_1$ and $C_2$ are two curve in $\mathbb{BC}$ with $-\infty\leq t_1,t_2\leq1$ and $-\infty\leq t_1,t_2\leq0$ respectively. Setting $t=1-\frac{s}{|Z|_h}$ in the first integral and $t=-\frac{s}{|Z|_h}$ in the second integral, where $s=s_1e_1+s_2e_2$, we obtain
 \begin{align}\label{eq:m3}
      \mathcal{J}_\mathcal{V}(Z)&=\frac{\exp((1-i)|Z|_h)}{2^\mathcal{V}|Z|_h^{\frac{1}{2}}\Gamma_b(\mathcal{V}+1)}\frac{\Gamma_b(2\mathcal{V}+1)}{\left(\Gamma_b(\mathcal{V}+\frac{1}{2})\right)^2}\left[\int_{C_1}\left(s\right)^{\mathcal{V}-\frac{1}{2}}\left(1-\frac{s}{|Z|_h}\right)^{\mathcal{V}-\frac{1}{2}}\exp(-s)ds\right]\nonumber\\
      &+\frac{\exp(-i|Z|_h)}{2^\mathcal{V} (-1)^{\mathcal{V}+\frac{1}{2}}|Z|_h^{\frac{1}{2}}\Gamma_b(\mathcal{V}+1)}\frac{\Gamma_b(2\mathcal{V}+1)}{\left(\Gamma_b(\mathcal{V}+\frac{1}{2})\right)^2}\left[\int_{C_2}(-s)^{\mathcal{V}+\frac{1}{2}}(1+\frac{s}{|Z|_h})^{\mathcal{V}-\frac{1}{2}}\exp(-s)ds\right]
 \end{align}
 For large value of $|Z|_h$ the second term in \eqref{eq:m3} negligible compared to the first term. Therefore,
 \begin{align}\label{eq:n3}
     \mathcal{J}_\mathcal{V}(Z)\thicksim \frac{\exp((1-i)|Z|_h)}{2^\mathcal{V}|Z|_h^{\frac{1}{2}}\Gamma_b(\mathcal{V}+1)}\frac{\Gamma_b(2\mathcal{V}+1)}{\left(\Gamma_b(\mathcal{V}+\frac{1}{2})\right)^2}\left[\int_{C_1}\left(s\right)^{\mathcal{V}-\frac{1}{2}}\left(1-\frac{s}{|Z|_h}\right)^{\mathcal{V}-\frac{1}{2}}\exp(-s)ds\right]
 \end{align}
Now we apply the Taylor formula of $(1-z)^a$ on idempotent components of $\left(1-\frac{s}{|Z|_h}\right)^{\mathcal{V}-\frac{1}{2}}$, we have
 \begin{align}\label{eq:c7}
     \left(1-\frac{s}{|Z|_h}\right)^{\mathcal{V}-\frac{1}{2}}&=\left(1-\frac{s_1}{|z_1|}\right)^{\nu_1-\frac{1}{2}}e_1+\left(1-\frac{s_2}{|z_2|}\right)^{\nu_2-\frac{1}{2}}e_2\nonumber\\
     &=\left[\sum_{k=0}^{n-1}\frac{\left(-\nu_1+\frac{1}{2}\right)_k}{k!|z_1|^k}s_1^k+\frac{\left(-\nu_1+\frac{1}{2}\right)_n}{n!|z_1|^n}\left(1-\frac{s_1'}{|z_1|}\right)^{\nu_1-\frac{1}{2}-n}s_1^n\right]e_1\nonumber\\
     &+\left[\sum_{k=0}^{n-1}\frac{\left(-\nu_2+\frac{1}{2}\right)_k}{k!|z_2|^k}s_2^k+\frac{\left(-\nu_2+\frac{1}{2}\right)_n}{n!|z_1|^n}\left(1-\frac{s_2'}{|z_2|}\right)^{\nu_2-\frac{1}{2}-n}s_2^n\right]e_2\nonumber\\
     &=\left[\sum_{k=0}^{n-1}\frac{\left(-\mathcal{V}+\frac{1}{2}\right)_k}{k!|Z|_h^k}s^k+\frac{\left(-\mathcal{V}+\frac{1}{2}\right)_n}{n!|Z|_h^n}\left(1-\frac{s'}{|Z|_h}\right)^{\nu_1-\frac{1}{2}-n}s^n\right],
 \end{align}
 where $s'=s'_1e_1+s'_2e_2$ which satisfies $0<s'_1<s_1$ and $0<s'_2<s_2$. Setting \eqref{eq:c7} in \eqref{eq:n3}, we get
 \begin{align}
      \mathcal{J}_\mathcal{V}(Z)&\thicksim \frac{\exp((1-i)|Z|_h)}{2^\mathcal{V}|Z|^{\frac{1}{2}}\Gamma_b(\mathcal{V}+1)}\frac{\Gamma_b(2\mathcal{V}+1)}{\left(\Gamma_b(\mathcal{V}+\frac{1}{2})\right)^2}\left[\sum_{k=0}^{n-1}\frac{\left(-\mathcal{V}+\frac{1}{2}\right)_k}{k!|Z|_h^k}\int_{C_1}\left(s\right)^{\mathcal{V}+k-\frac{1}{2}}\exp(-s)ds+R_n\right]\nonumber\\
      &=\frac{\exp\left((1-i)|Z|_h\right)}{2^\mathcal{V}|Z|_h^{\frac{1}{2}}\Gamma_b(\mathcal{V}+1)}\frac{\Gamma_b(2\mathcal{V}+1)}{\left(\Gamma_b(\mathcal{V}+\frac{1}{2})\right)^2}\left[\sum_{k=0}^{n-1}\frac{\left(-\mathcal{V}+\frac{1}{2}\right)_k}{k!|Z|_h^k}\Gamma_b\left(\mathcal{V}+k+\frac{1}{2}\right)+R_n\right]\nonumber,
 \end{align}
 where
 \begin{align*}
     R_n=\frac{\left(-\mathcal{V}+\frac{1}{2}\right)_n}{n!|Z|_h^n}\int_{C_1}\left(s\right)^{\mathcal{V}+n-\frac{1}{2}}\left(1-\frac{s'}{|Z|_h}\right)^{\mathcal{V}-\frac{1}{2}-n}\exp(-s)ds.
 \end{align*}
 For large value of $|Z|_h$ and $\Re(\alpha_1)-n-\frac{1}{2}>\Im(\alpha_2)$, it follows that $\left(1-\frac{s'}{Z}\right)^{\mathcal{V}-n-\frac{1}{2}}<_h1$. Therefore,
\begin{align*}
    |R_n|_h<_h\left|\frac{\left(-\mathcal{V}+\frac{1}{2}\right)_n}{n!|Z|_h^n}\Gamma_b\left(\mathcal{V}+n+\frac{1}{2}\right)\right|_h \rightarrow 0,\; \text{as}\;|Z|_h\rightarrow \infty.
\end{align*}
Therefore,
\begin{align*}
    \mathcal{J}_\mathcal{V}(Z)&\thicksim\frac{\exp((1-i)|Z|_h)}{2^\mathcal{V}|Z|_h^{\frac{1}{2}}\Gamma_b(\mathcal{V}+1)}\frac{\Gamma_b(2\mathcal{V}+1)}{\left(\Gamma_b(\mathcal{V}+\frac{1}{2})\right)^2}\sum_{k=0}^{n-1}\frac{\left(-\mathcal{V}+\frac{1}{2}\right)_k}{k!|Z|_h^k}\Gamma_b(\mathcal{V}+k+\frac{1}{2}),\;\text{as}\;|Z|_h\rightarrow \infty.
\end{align*}
\end{proof}
\begin{theo}\label{th:w4}
     If $Z=z_1e_1+z_2e_2, \mathcal{V}=\alpha_1 +j\alpha_2=\nu_1e_1+\nu_2e_2 \in{\mathbb{BC}}$ where $e_1=\frac{1+k}{2}, e_2=\frac{1-k}{2}$, then for a fixed $Z\notin\mathbb{O}_2$, bicomplex Bessel function $\mathcal{J}_\mathcal{V}(Z)$ is a holomorphic function of $\mathcal{V}$.
\end{theo}
\begin{proof}
    Based on the idempotent representation, the bicomplex Bessel function is expressed as,
\begin{align}\label{eq:18}
    \mathcal{J}_\mathcal{V}(Z)= J_{\nu_1}(z_1)e_1+ J_{\nu_2}(z_2)e_2.
\end{align}
Setting $e_1=\frac{1+ij}{2}$ and $e_2=\frac{1-ij}{2}$ in \eqref{eq:18}, we get
\begin{align}\label{eq:19}
      \mathcal{J}_\mathcal{V}(Z)&= J_{\nu_1}(z_1)\left(\frac{1+ij}{2}\right)+ J_{\nu_2}(z_2)\left(\frac{1-ij}{2}\right)\nonumber\\
     &=\frac{1}{2}( J_{\nu_1}(z_1)+J_{\nu_2}(z_2))+\frac{ij}{2}(J_{\nu_1}(z_1)-J_{\nu_2}(z_2)).
\end{align}
Let us consider $  \mathcal{J}_\mathcal{V}(Z)=g_1(\alpha_1,\alpha_2)+jg_2(\alpha_1,\alpha_2)$. Then from the equation \eqref{eq:19}, we obtain
\begin{align*}
    &g_1(\alpha_1,\alpha_2)=\frac{1}{2}\left( J_{\nu_1}(z_1)+J_{\nu_2}(z_2)\right)\\
   &g_2(\alpha_1,\alpha_2)= \frac{i}{2}\left(J_{\nu_1}(z_1)-J_{\nu_2}(z_2)\right).
\end{align*}
Now,
\begin{align*}
   & \frac{\partial g_1}{\partial \alpha_1}=\frac{1}{2}( J'_{\nu_1}(z_1)+J'_{\nu_2}(z_2)) \\
    & \frac{\partial g_1}{\partial \alpha_2}=\frac{-i}{2}\left(J'_{\nu_1}(z_1)\right) + \frac{i}{2}\left(J'_{\nu_2}(z_2)\right)\\
    &\frac{\partial g_2}{\partial \alpha_1}=\frac{i}{2}\left(J'_{\nu_1}(z_1)\right) - \frac{i}{2}\left(J'_{\nu_2}(z_2)\right)\\
    &\frac{\partial g_2}{\partial \alpha_2}=\frac{i}{2}\left(J'_{\nu_1}(z_1)\right)\cdot(-i) - \frac{i}{2}\left(J'_{\nu_2}(z_2)\right)\cdot i\\
    &\hspace{20pt}= \frac{1}{2}\left(J'_{\nu_1}(z_1)+J'_{\nu_2}(z_2)\right)
\end{align*}
from the above equation, we get
\begin{align*}
    \frac{\partial g_1}{\partial \alpha_1}= \frac{\partial g_2}{\partial \alpha_2} \quad
      and \quad  \frac{\partial g_1}{\partial \alpha_2}=- \frac{\partial g_2}{\partial \alpha_1}.
\end{align*}
Therefore, $g_1$ and $g_2$ satisfies bicomplex Cauchy-Riemann equation. Again, $J_{\nu_1}(z_1)$ and $J_{\nu_2}(z_2)$ are complex Bessel function , then for a fixed $Z\notin\mathbb{O}_2$, $g_1$ and $g_2$ are analytic functions of $\alpha_1$ and $\alpha_2$. From the analyticity condition of the function of bicomplex variable \eqref{eq:i5} we obtain, bicomplex Bessel function $ \mathcal{J}_\mathcal{V}(Z)$ is a holomorphic function of $\mathcal{V}$.
\end{proof}
\begin{theo}
    Assume that $Z=z_1e_1+z_2e_2, \mathcal{V}=\alpha_1 +j\alpha_2=\nu_1e_1+\nu_2e_2 \in{\mathbb{BC}}$, with satisfies the conditions $\alpha_1=\frac{l_1+l_2}{2}$ and $\alpha_2=\frac{l_2-l_1}{2i}$ where $l_1,l_2\in \mathbb{N}\cup \{0\}$, then bicomplex Bessel function $ \mathcal{J}_\mathcal{V}(Z)$ is a holomorphic function of $Z$ in the whole bicomplex space.
\end{theo}
\begin{proof}
    The proof of this theorem follows a similar approach as used in Theorem \ref{th:w4}.
\end{proof}
\section{The bicomplex testing spaces and n-dimensional bicomplex Hankel transformation}
In 1966, A. H. Zemanian \cite{zehankel} has studied $\mu$th order Hankel transformation in the test function space $H_\mu$. The $\mu$th order Hankel transformation is defined by
\begin{align}\label{eq:w6}
    h_\mu\psi(y)=\int_{0}^{\infty}\psi(x)\sqrt{xy}J_{\mu}(xy)dx,
\end{align}
where $\psi\in H_\mu$ and $J_{\mu}$ designates the Bessel function of order $\mu$. For $\mu\geq-\frac{1}{2}$ the $\mu$th order Hankel transforms is an automorphism on the space $H_\mu$. Furthermore E. L. Koh \cite{hankel} generalized the $\mu$th order Hankel transform \eqref{eq:w6} in $n$-dimensional is defined by
    $$\mathfrak{h}_\mu\psi(y_1,...,y_n)=\int_{0}^{\infty}\ldots\int_{0}^{\infty}\psi(x_1,...,x_n)\prod_{i=1}^{n}\sqrt{x_iy_i}J_\mu(x_iy_i)dx_1\ldots dx_n.$$
  In this section, we define $n$-dimensional bicomplex Hankel transformation by using bicomplex Bessel function and establish some properties on the bicomplex testing space $\mathcal{A}_{\sigma,\varepsilon}$. We use the following notations $E=\{\omega\in\mathbb{R}^n:0<\omega_i<\infty,i=1,2,...,n\}$ where $\omega=(\omega_1,\omega_2,...\omega_n)$ and $[\omega]$ refers to the product $\omega_1\omega_2\ldots\omega_n$. Similarly, $[\omega^p]=\omega_1^{p_1}\omega^{p_2}\ldots\omega^{p_n}$, where $p=(p_1,p_2,...,p_n)$. We use the symbols $(p)=p_1+p_2+\ldots+p_n$,
 \begin{align*}
     D_\omega^{(p)}=\frac{\partial^{(p)}}{\partial\omega_1^{p_1}\partial\omega_2^{p_2}\ldots\partial\omega_n^{p_n}}\;\;\mbox{and} \;\;(\omega^{-1}D_\omega)^p=\prod_{l=1}^{n}\left(\frac{1}{\omega_l}\frac{\partial}{\partial\omega_l}\right)^{p_l},
 \end{align*}
 where $p=(p_1,p_2,...,p_n)$ is a nonnegative integer in $\mathbb{R}^n$. Let $\varepsilon$ be a positive real number, now we define bicomplex testing function space $\mathcal{A}_{\sigma,\varepsilon}$ for every fixed hyperbolic number $\sigma=\sigma_1e_1+\sigma_2e_2\in\mathbb{BC}$ as follows,
the space of bicomplex valued functions $\zeta(\omega)$ which is defined and smooth on $E$, $\zeta(\omega)=0$ on $E_1=\{\omega\in\mathbb{R}^n:\varepsilon<\omega_i<\infty,i=1,2,...,n\}$ and for each nonnegative integers $p$ in $\mathbb{R}^n$
\begin{align*}
    \mathcal{T}^{\sigma,\varepsilon}_{p}(\zeta(\omega))=\sup\limits_{E}\left|(\omega^{-1}D_\omega)^p[\omega]^{-\sigma-\frac{1}{2}}\zeta(\omega)\right|_h<_h\infty.
\end{align*}
Clearly, $\mathcal{A}_{\sigma,\varepsilon}$ is a vector space over the field of complex numbers.
\begin{theo}
    Suppose that $\lambda=s_1+js_2=\lambda_1e_1+\lambda_2e_2,e_1=\frac{1+k}{2}, e_2=\frac{1-k}{2}\in\mathbb{BC}$ with satisfies $s_1=n_1+n_2$ and $s_2=i(n_1-n_2)$ where $n_1,n_2\in \mathbb{N}$, then $\mathcal{A}_{\sigma+\lambda,\varepsilon}\subset \mathcal{A}_{\sigma,\varepsilon}$.
\end{theo}
\begin{proof}
    Let a bicomplex valued function $\zeta(\omega)=\zeta_1(\omega)e_1+\zeta_2(\omega)e_2\in\mathcal{A}_{\sigma+\lambda,\varepsilon}$, then
    \begin{align}\label{eq:v9}
       &\mathcal{T}^{\sigma+\lambda,\varepsilon}_{p}(\zeta(\omega))=\sup\limits_{E}\left|(\omega^{-1}D_\omega)^p[\omega]^{-\sigma-\lambda-\frac{1}{2}}\zeta(\omega)\right|_h<_h\infty\nonumber\\
       &\implies \sup\limits_{E}\left|(\omega^{-1}D_\omega)^p[\omega]^{-\sigma_1-\lambda_1-\frac{1}{2}}\zeta_1(\omega)\right|e_1+\sup\limits_{E}\left|(\omega^{-1}D_\omega)^p[\omega]^{-\sigma_2-\lambda_2-\frac{1}{2}}\zeta_2(\omega)\right|e_2<_h\infty.
    \end{align}
    Let us take $\mathcal{K}_{1,p}^{\sigma_1+\lambda_1,\varepsilon}(\zeta_1)=\sup\limits_{E}\left|(\omega^{-1}D_\omega)^p[\omega]^{-\sigma_1-\lambda_1-\frac{1}{2}}\zeta_1(\omega)\right|$ and \\$\mathcal{K}_{2,p}^{\sigma_2+\lambda_2,\varepsilon}(\zeta_2)=\sup\limits_{E}\left|(\omega^{-1}D_\omega)^p[\omega]^{-\sigma_2-\lambda_2-\frac{1}{2}}\zeta_2(\omega)\right|$, then from \eqref{eq:v9}, we have $\mathcal{K}_{1,p}^{\sigma_1+\lambda_1,\varepsilon}<\infty$ and $\mathcal{K}_{2,p}^{\sigma_2+\lambda_2,\varepsilon}<\infty$.
    Now, it is easy to show that
    \begin{align}\label{eq:f74}
        (\omega^{-1}D_\omega)^p[\omega]^{-\sigma_1-\frac{1}{2}}\zeta_1(\omega)&=\prod_{k=1}^{n}\left(1+\frac{\omega_k}{2p_k}\frac{\partial}{\partial\omega_k}\right)\left[2p_k\left(\frac{1}{\omega_k}\frac{\partial}{\partial\omega_k}\right)^{p_k-1}\omega_k^{-\sigma_1-\frac{5}{2}}\zeta_1(\omega)\right]\nonumber\\
        &=\prod_{k=1}^{n}\left(1+\frac{\omega_k^2}{2p_k}\frac{1}{\omega_k}\frac{\partial}{\partial\omega_k}\right)\left[2p_k\left(\frac{1}{\omega_k}\frac{\partial}{\partial\omega_k}\right)^{p_k-1}\omega_k^{-\sigma_1-\frac{5}{2}}\zeta_1(\omega)\right].
    \end{align}
    Since  $\zeta(\omega)=0$ on $E_1=\{\omega\in\mathbb{R}^n:\varepsilon<\omega_i<\infty,i=1,2,...,n\}$ and it implies that $\zeta_1(\omega)=0$ on $E_1$. After simplifying the right-hand side of \eqref{eq:f74} , we get to the sum of $2^n$ components
    \begin{align}\label{eq:9d}
    \mathcal{K}_{1,p}^{\sigma_1,\varepsilon}(\zeta_1)\leq 2^n[p]\mathcal{K}_{1,p-n}^{\sigma_1+2,\varepsilon}(\zeta_1)+c_2\mathcal{K}_{1,p-n+1}^{\sigma_1+2,\varepsilon}(\zeta_1)+\ldots+\epsilon^{2n}\mathcal{K}_{1,p}^{\sigma_1+2,\varepsilon}(\zeta_1)<\infty\end{align}
    Again,using the similar procedure we obtain
    \begin{align}\label{eq:c8}
    \mathcal{K}_{2,p}^{\sigma_2,\varepsilon}(\zeta_2)&\leq 2^n[p]\mathcal{K}_{2,p-n}^{\sigma_2+2,\varepsilon}(\zeta_2)+c_2\mathcal{K}_{2,p-n+1}^{\sigma_2+2,\varepsilon}(\zeta_2)+\ldots+\epsilon^{2n}\mathcal{K}_{2,p}^{\sigma_2+2,\varepsilon}(\zeta_2)\nonumber\\&\leq d_1\mathcal{K}_{2,p-2n}^{\sigma_2+4,\varepsilon}(\zeta_2)+d_2\mathcal{K}_{2,p-2n+1}^{\sigma_2+4,\varepsilon}(\zeta_2)+\ldots+d_k\mathcal{K}_{2,p+n}^{\sigma_2+4,\varepsilon}(\zeta_2)<\infty,
    \end{align}
    where $c_1,c_2,...,c_l,d_1,d_2,...,d_k$ are constants. If we take $\lambda_1=2$ and $\lambda_2=4$, the by using \eqref{eq:9d} and \eqref{eq:c8}, we have
    \begin{align*}
      \mathcal{T}^{\sigma,\varepsilon}_{p}(\zeta(\omega))&= \mathcal{K}_{1,p}^{\sigma_1,\varepsilon}(\zeta_1)e_1+ \mathcal{K}_{2,p}^{\sigma_2,\varepsilon}(\zeta_2)e_2\\
      &<_h \infty.
    \end{align*}
    Therefore, $\mathcal{A}_{\sigma+2e_1+4e_2,\varepsilon}\subset \mathcal{A}_{\sigma,\varepsilon}$. Using the similar procedure we can easily prove that $\mathcal{A}_{\sigma+\lambda,\varepsilon}\subset \mathcal{A}_{\sigma,\varepsilon}$.
\end{proof}
 Let us consider $\rho$ be a positive hyperbolic number, $\Omega$ be an open subset of $\mathbb{BC}^n$ and $\mathcal{Z}=(Z_1,Z_2,...,Z_n)\in\Omega$. We define another testing space $\mathcal{C}^\mathcal{V}_\rho$ as follows, a bicomplex valued infinitely differentiable function $\Psi(\mathcal{Z})\in \mathcal{C}^\mathcal{V}_\rho$ if and only if for each nonnegative integer $m$
 \begin{align*}
 \mathcal{G}^\mathcal{V}_{\rho,m}(\Psi(\mathcal{Z}))=\sup_{\mathcal{Z}}\left|\exp\left(-\rho\sum_{k=1}^{n}|Z_k|_h\right)[\mathcal{Z}]^{2m-\mathcal{V}-\frac{1}{2}}\Psi(\mathcal{Z})\right|_h<_h\infty.
 \end{align*}
 It is evident that $\mathcal{C}^\mathcal{V}_\rho$
  forms a vector space over the complex field.
 \begin{theo}
     If $a$ and $b$ are two positive hyperbolic numbers with satisfies the condition $a>_hb$, then $\mathcal{C}^\mathcal{V}_b\subset\mathcal{C}^\mathcal{V}_a$.
 \end{theo}
 \begin{proof}
     Let us consider the idempotent form of bicomplex numbers  $a=a_1e_1+a_2e_2,b=b_1e_1+b_2e_2,\mathcal{V}=\nu_1e_1+\nu_2e_2$ and  $Z_i=z_{1i}e_1+z_{2i}e_2$ for $i=1,2,...,n$. Then $[\mathcal{Z}]=\prod\limits_{i=1}^{n}(z_{1i})e_1+\prod\limits_{i=1}^{n}(z_{2i})e_2$ and $\sum\limits_{k=1}^{n}|Z_k|_h=\sum\limits_{k=1}^{n}|z_{1k}|e_1+\sum\limits_{k=1}^{n}|z_{2k}|e_2$. Again from the condition $a>_hb$, we obtain $a_1>b_1$ and $a_2>b_2$. Now we consider $\Psi(\mathcal{Z})=\Psi_1(Z_1,Z_2,...,Z_n)e_1+\Psi_2(Z_1,Z_2,...,Z_n)e_2\in\mathcal{C}^\mathcal{V}_b$, then
     \begin{align*}
         \mathcal{G}^\mathcal{V}_{b,m}(\Psi(\mathcal{Z}))=\sup_{\mathcal{Z}}\left|\exp\left(-b\sum_{k=1}^{n}|Z_k|_h\right)[\mathcal{Z}]^{2m-\mathcal{V}-\frac{1}{2}}\Psi(\mathcal{Z})\right|_h<_h\infty.
     \end{align*}
     Since $a_1>b_1$ and $a_2>b_2$, therefore the inequalities
     \begin{align}\label{eq:d9}
         &\sup_{\mathcal{Z}}\left|\exp\left(-a_i\sum\limits_{k=1}^{n}|z_{ik}|\right)\left[\prod\limits_{k=1}^{n}|z_{ik}|\right]^{2m-\nu_i-\frac{1}{2}}\Psi_i(\mathcal{Z})\right|\nonumber\\&\hspace{10pt}< \sup_{\mathcal{Z}}\left|\exp\left(-b_i\sum\limits_{k=1}^{n}|z_{ik}|\right)\left[\prod\limits_{k=1}^{n}|z_{ik}|\right]^{2m-\nu_i-\frac{1}{2}}\Psi_i(\mathcal{Z})\right|,
     \end{align}
     holds for $i=1,2$. Using the inequalities \eqref{eq:d9}, we obtain
     \begin{align*}
          \mathcal{G}^\mathcal{V}_{a,m}(\Psi(\mathcal{Z}))&=\sup_{\mathcal{Z}}\left|\exp\left(-a\sum_{k=1}^{n}|Z_k|_h\right)[\mathcal{Z}]^{2m-\mathcal{V}-\frac{1}{2}}\Psi(\mathcal{Z})\right|\\
          &=\sum_{i=1}^{2}\sup_{\mathcal{Z}}\left|\exp\left(-a_i\sum\limits_{k=1}^{n}|z_{ik}|\right)\left[\prod\limits_{k=1}^{n}|z_{ik}|\right]^{2m-\nu_i-\frac{1}{2}}\Psi_i(\mathcal{Z})\right|e_i\\
          &<_h\sum_{i=1}^{2}\sup_{\mathcal{Z}}\left|\exp\left(-b_i\sum\limits_{k=1}^{n}|z_{ik}|\right)\left[\prod\limits_{k=1}^{n}|z_{ik}|\right]^{2m-\nu_i-\frac{1}{2}}\Psi_i(\mathcal{Z})\right|e_i\\
          &=\sup_{\mathcal{Z}}\left|\exp\left(-b\sum_{k=1}^{n}|Z_k|_h\right)[\mathcal{Z}]^{2m-\mathcal{V}-\frac{1}{2}}\Psi(\mathcal{Z})\right|\\
         &<_h\infty.
     \end{align*}
     Therefore, $\mathcal{C}^\mathcal{V}_b\subset\mathcal{C}^\mathcal{V}_a$. Hence, the proof is completed.
 \end{proof}
 We now define several bicomplex differential operators are as follows:
 \begin{align*}
     &\mathcal{N}_{k\sigma}=\sum_{l=1}^2\omega_k^{\sigma_l+\frac{1}{2}}\frac{\partial}{\partial\omega_k}\omega_k^{-\sigma_l-\frac{1}{2}}e_l=\omega_k^{\sigma+\frac{1}{2}}\frac{\partial}{\partial\omega_k}\omega_k^{-\sigma-\frac{1}{2}}\\
     &\mathcal{M}_{k\sigma}=\sum_{l=1}^2\omega_k^{-\sigma_l-\frac{1}{2}}\frac{\partial}{\partial\omega_k}\omega_k^{\sigma_l+\frac{1}{2}}e_l=\omega_k^{-\sigma-\frac{1}{2}}\frac{\partial}{\partial\omega_k}\omega_k^{\sigma+\frac{1}{2}}\\
     &\mathcal{N}_{\sigma}=\prod_{k=1}^{n}\mathcal{N}_{k\sigma}=[\omega]^{\sigma+\frac{1}{2}}\frac{\partial^n}{\partial\omega_1...\partial\omega_n}[\omega]^{-\sigma-\frac{1}{2}}\\
     &\mathcal{M}_{\sigma}=\prod_{k=1}^{n}\mathcal{M}_{k\sigma}=[\omega]^{-\sigma-\frac{1}{2}}\frac{\partial^n}{\partial\omega_1...\partial\omega_n}[\omega]^{\sigma+\frac{1}{2}},
 \end{align*}
 where $e_1=\frac{1+k}{2}$ and $e_2=\frac{1-k}{2}$.
 \begin{theo}
 Let $\sigma=\sigma_1e_1+\sigma_2e_2$ be hyperbolic number and $e_1=\frac{1+k}{2}, e_2=\frac{1-k}{2}\in\mathbb{BC}$, then
 \begin{description}
    \item[(i)] The operation $\zeta(\omega)\rightarrow \mathcal{N}_{\sigma}\zeta(\omega)$ is a continuous linear mapping of $\mathcal{A}_{\sigma,\varepsilon}$ into $\mathcal{A}_{\sigma+1,\varepsilon}$.
     \item [(ii)] The operation $\zeta(\omega)\rightarrow \mathcal{M}_{\sigma}\zeta(\omega)$ is a continuous linear mapping of $\mathcal{A}_{\sigma+1,\varepsilon}$ into $\mathcal{A}_{\sigma,\varepsilon}$.
\end{description}
 \end{theo}
 \begin{proof}
    (i) Let us consider a bicomplex valued function  $\zeta(\omega)\in\mathcal{A}_{\sigma,\varepsilon}$. Here we have used the concept of Lemma \ref{lemma2}. Now,
     \begin{align*}
         \mathcal{T}^{\sigma+1,\varepsilon}_{p}(\mathcal{N}_{\sigma}\zeta(\omega))&=\sup\limits_{E}\left|(\omega^{-1}D_\omega)^p[\omega]^{-\sigma-1-\frac{1}{2}}\mathcal{N}_{\sigma}\zeta(\omega)\right|_h\\
         &=\sup\limits_{E}\left|(\omega^{-1}D_\omega)^p[\omega]^{-1}\frac{\partial^n}{\partial\omega_1...\partial\omega_n}[\omega]^{-\sigma-\frac{1}{2}}\zeta(\omega)\right|_h\\
        &=\sup\limits_{E}\left|(\omega^{-1}D_\omega)^{p+n}[\omega]^{-\sigma-\frac{1}{2}}\zeta(\omega)\right|_h\\
        &=\mathcal{T}^{\sigma,\varepsilon}_{p+n}(\zeta(\omega)).
        \end{align*}
    Therefore $\mathcal{N}_{\sigma}$ is continuous linear mapping from $\mathcal{A}_\sigma$ into $\mathcal{A}_{\sigma+1}$.\\
    (ii) Let $p=(p_1,...,p_n)$ be a nonnegative integer in $\mathbb{R}^n$ and $\zeta(\omega)\in\mathcal{A}_{\sigma+1,\varepsilon}$. Then
    \begin{align}\label{eq:z8}
         \mathcal{T}^{\sigma,\varepsilon}_{p}(\mathcal{M}_{\sigma}\zeta(\omega))&=\sup\limits_{E}\left|(\omega^{-1}D_\omega)^p[\omega]^{-\sigma-\frac{1}{2}}\mathcal{M}_{\sigma}\zeta(\omega)\right|_h\nonumber\\
         &=\sup\limits_{E}\left|(\omega^{-1}D_\omega)^p[\omega]^{-2\sigma-1}\frac{\partial^n}{\partial\omega_1...\partial\omega_n}[\omega]^{\sigma+\frac{1}{2}}\zeta(\omega)\right|_h\nonumber\\
         &=\sup\limits_{E}\left|(\omega^{-1}D_\omega)^p[\omega]^{-2\sigma-1}\frac{\partial^n}{\partial\omega_1...\partial\omega_n}[\omega]^{2\sigma+2}[\omega]^{-\sigma-\frac{3}{2}}\zeta(\omega)\right|_h.
         \end{align}
    It is clear to observe that
    \begin{align}\label{eq:6s}
        \frac{\partial^n}{\partial\omega_1...\partial\omega_n}[\omega]^{2\sigma+2}[\omega]^{-\sigma-\frac{3}{2}}\zeta(\omega)&=(2\sigma+2)^n[\omega]^{2\sigma+1}[\omega]^{-\sigma-\frac{3}{2}}\zeta(\omega)\nonumber\\&+(2\sigma+2)^{n-1}[\omega]^{2\sigma+1}\sum_{i=1}^{n}\omega_i\frac{\partial}{\partial\omega_i}[\omega]^{-\sigma-\frac{3}{2}}\zeta(\omega)\nonumber\\
        &+(2\sigma+2)^{n-2}[\omega]^{2\sigma+1}\sum_{i,j=1}^{n}\omega_i\omega_j\frac{\partial^2}{\partial\omega_i\partial\omega_j}[\omega]^{-\sigma-\frac{3}{2}}\zeta(\omega)\nonumber\\
        &+\ldots+[\omega]^{2\sigma+2}\frac{\partial^n}{\partial\omega_1...\partial\omega_n}[\omega]^{-\sigma-\frac{3}{2}}\zeta(\omega).
    \end{align}
    Now using \eqref{eq:6s}, we get
    \begin{align}\label{eq:3f}
        (\omega^{-1}D_\omega)^p[\omega]^{-2\sigma-1}\frac{\partial^n}{\partial\omega_1...\partial\omega_n}[\omega]^{2\sigma+2}[\omega]^{-\sigma-\frac{3}{2}}\zeta(\omega)&=(2\sigma+2)^n (\omega^{-1}D_\omega)^p[\omega]^{-\sigma-\frac{3}{2}}\zeta(\omega)\nonumber\\&+(2\sigma+2)^{n-1}\sum_{i=1}^{n} (\omega^{-1}D_\omega)^p\omega_i\frac{\partial}{\partial\omega_i}[\omega]^{-\sigma-\frac{3}{2}}\zeta(\omega)\nonumber\\
        &+\ldots+(\omega^{-1}D_\omega)^p[\omega]\frac{\partial^n}{\partial\omega_1...\partial\omega_n}[\omega]^{-\sigma-\frac{3}{2}}\zeta(\omega).
    \end{align}
    Again,
    \begin{align}\label{eq:6g}
       (\omega^{-1}D_\omega)^p\omega_i\frac{\partial}{\partial\omega_i}[\omega]^{-\sigma-\frac{3}{2}}\zeta(\omega)&=(\omega^{-1}D_\omega)^p\omega_i^2\frac{1}{\omega_i}\frac{\partial}{\partial\omega_i}[\omega]^{-\sigma-\frac{3}{2}}\zeta(\omega)\nonumber\\
       &=\left(2p_i+\omega_i\frac{\partial}{\partial\omega_i}\right)(\omega^{-1}D_\omega)^p[\omega]^{-\sigma-\frac{3}{2}}\zeta(\omega).
    \end{align}
    Using similar procedure we can easily prove
    \begin{align}\label{eq:9h}
        (\omega^{-1}D_\omega)^p[\omega]\frac{\partial^n}{\partial\omega_1...\partial\omega_n}[\omega]^{-\sigma-\frac{3}{2}}\zeta(\omega)=\prod_{i=1}^n\left(2p_i+\omega_i\frac{\partial}{\partial\omega_i}\right)(\omega^{-1}D_\omega)^p[\omega]^{-\sigma-\frac{3}{2}}\zeta(\omega)
    \end{align}
    Using \eqref{eq:z8},\eqref{eq:3f},\eqref{eq:6g} and \eqref{eq:9h}, we obtain
    \begin{align*}
        \mathcal{T}^{\sigma,\varepsilon}_{p}(\mathcal{M}_{\sigma}\zeta(\omega))&<_h(2\sigma+2)^n \mathcal{T}^{\sigma+1,\varepsilon}_{p}(\zeta(\omega))\nonumber\\&+(2\sigma+2)^{n-1}\sum_{i=1}^{n} \left(2p_i\mathcal{T}^{\sigma+1,\varepsilon}_{p}(\zeta(\omega))+\varepsilon^2\mathcal{T}^{\sigma+1,\varepsilon}_{p'}(\zeta(\omega))\right)\nonumber\\
        &+\ldots+\sum_{i=0}^{n}s_i\varepsilon^{2i}\mathcal{T}^{\sigma+1,\varepsilon}_{p'_i}(\zeta(\omega)),
    \end{align*}
    where $p',p'_i$ are nonnegative integer in $\mathbb{R}^n$ and $s_i$ are constants for $i=1,2,...,n$. Hence, $\mathcal{M}_{\sigma}$ is continuous linear mapping from   $\mathcal{A}_{\sigma+1,\varepsilon}$ into $\mathcal{A}_{\sigma,\varepsilon}$.
\end{proof}

 Let $\mathcal{V}=\alpha_1+j\alpha_2=\nu_1e_1+\nu_2e_2\in\mathbb{BC},\rho=\theta_1+j\theta_2=\rho_1e_1+\rho_2e_2$ be a positive hyperbolic number, $\varrho=\min\{\rho_1,\rho_2\}$ and $\mathcal{V}$ satisfies the condition $\Re(\alpha_1)+\frac{1}{2}\geq\Im(\alpha_2)$. We define $n$-dimensional bicomplex Hankel transformation $\mathscr{H}_\mathcal{V}$ is given by
 \begin{align}\label{eq:h11}
\eta(\mathcal{Z})=\mathscr{H}_\mathcal{V}(\zeta)=\int_{0}^{\infty}\ldots\int_{0}^{\infty}\zeta(\omega) \prod_{k=1}^{n}\sqrt{\omega_kZ_k}\;\mathcal{J}_\mathcal{V}\left(\omega_kZ_k\right) d\omega_1\ldots d\omega_n,
 \end{align}
 where, $\mathcal{Z}\in{\mathbb{S}}=\{(Z_1,...,Z_n):Z_k=z_{1k}e_1+z_{2k}e_2,|\Im(z_{1k})|<\varrho,|\Im(z_{2k})|<\varrho\;\mbox{and}\;z_{1k},z_{2k}\notin(-\infty,0]\}$ and $\zeta\in\mathcal{A}_{\rho}$.
 \begin{theo}
     Assume that $\varrho\in\mathbb{R}^+,\mathcal{V}=\alpha_1+j\alpha_2=\nu_1e_1+\nu_2e_2\in\mathbb{BC}$, $\rho=\rho_1e_1+\rho_2e_2$ be a positive hyperbolic number such that $\Re(\alpha_1)+\frac{1}{2}\geq|\Im(\alpha_2)|$ and $|\rho|_h<_h\varrho$. For $\mathcal{Z}\in\mathbb{S}$ and $0<\omega_k<\infty$,
     \begin{align}\label{eq:c96}
         \left|\exp\left(-\rho|Z_k|_h\right)(\omega_kZ_k)^{-\mathcal{V}}\mathcal{J}_{\mathcal{V}+2m}(\omega_kZ_k)\right|_h<_hA_k,\; \mbox{for}\; k=1,2,...,n.
     \end{align}
 \end{theo}
 \begin{proof}
From Theorem \ref{eq:s3}, we see that $(\omega_kZ_k)^{-\mathcal{V}}\mathcal{J}_{\mathcal{V}+2m}(\omega_kZ_k)$ is absolutely hyperbolic convergent in $\mathbb{BC}$ and the idempotent representation is given by
\begin{align}\label{eq:9e}
    (\omega_kZ_k)^{-\mathcal{V}}\mathcal{J}_{\mathcal{V}+2m}(\omega_kZ_k)=(\omega_kz_{1k})^{-\nu_1}J_{\nu_1+2m}(\omega_kz_{1k})e_1+(\omega_kz_{2k})^{-\nu_2}J_{\nu_2+2m}(\omega_kz_{2k})e_2.
\end{align}
Now applying the asymptotic formula of Bessel function as $|z|\rightarrow\infty$,
\begin{align*}
    J_{\nu}(z)\thicksim \frac{1}{\sqrt{2\pi z}}\left(e^{i(z-\frac{\nu\pi}{2}-\frac{\pi}{4})}+e^{-i(z-\frac{\nu\pi}{2}-\frac{\pi}{4})}\right),\;\;-\pi<\arg z\leq\pi,
\end{align*}
in \eqref{eq:9e}, we obtain the inequalities
\begin{align*}
    \left|(\omega_kz_{ik})^{-\nu_i}J_{\nu_i+2m}(\omega_kz_{ik})\right|&<A_{ik}|z_{ik}|^{-\nu_i-\frac{1}{2}}\left(e^{-\Im(z_{ik})+\frac{\Im(\nu_i)\pi}{2}}+e^{\Im(z_{ik})-\frac{\Im(\nu_i)\pi}{2}}\right)
\end{align*}
holds for $i=1,2$ where $A_{ik}$ are constants.
Since, $\mathcal{V}$ satisfies the inequalities  $\Re(\alpha_1)+\frac{1}{2}\geq|\Im(\alpha_2)|$ and it is implies that $\Re(\nu_1)\geq-\frac{1}{2}$ and $\Re(\nu_2)\geq-\frac{1}{2}$, then for $|z_{ik}|>1$ and $\mathcal{Z}\in\mathbb{S}$, we get
\begin{align}\label{eq:s5}
    \left|\exp\left(-\rho_i|z_{ik}|\right)(\omega_kz_{ik})^{-\nu_i}J_{\nu_i+2m}(\omega_kz_{ik})\right|<A_{ik},\;\mbox{for}\; i=1,2.
\end{align}
Using \eqref{eq:9e} and \eqref{eq:s5}, we have
\begin{align*}
    \left|\exp\left(-\rho|Z_k|_h\right)(\omega_kZ_k)^{-\mathcal{V}}\mathcal{J}_{\mathcal{V}+2m}(\omega_kZ_k)\right|_h&=\left|\exp\left(-\rho_1|z_{1k}|\right)(\omega_kz_{1k})^{-\nu_1}J_{\nu_1+2m}(\omega_kz_{1k})\right|e_1\\
    &+\left|\exp\left(-\rho_2|z_{2k}|\right)(\omega_kz_{2k})^{-\nu_2}J_{\nu_2+2m}(\omega_kz_{2k})\right|e_2\\
    &<_hA_k.
\end{align*}
Hence, the proof of this theorem is completed.
 \end{proof}
 \begin{theo}\label{th:f3}
     Suppose that $\zeta(\omega)\in\mathcal{A}_\sigma$ and $\mathcal{V}=\alpha_1+j\alpha_2=\nu_1e_1+\nu_2e_2\in\mathbb{BC}$, with satisfies the condition $\Re(\alpha_1)+\frac{1}{2}\geq|\Im(\alpha_2)|$, then
     \begin{description}
         \item[(i)]  $\mathscr{H}_{\mathcal{V}+\sigma+1}(\mathcal{N}_{\mathcal{V}+\sigma}\zeta(\omega))=(-1)^n[\mathcal{Z}]\mathscr{H}_{\mathcal{V}+\sigma}(\zeta(\omega))$
         \item[(ii)] $\mathscr{H}_{\mathcal{V}+\sigma}(\mathcal{M}_{\mathcal{V}+\sigma}\zeta(\omega))=[\mathcal{Z}]\mathscr{H}_{\mathcal{V}+\sigma+1}(\zeta(\omega))$
         \item[(iii)] $\mathscr{H}_{\mathcal{V}+\sigma}(\mathcal{M}_{\mathcal{V}+\sigma}\mathcal{N}_{\mathcal{V}+\sigma}\zeta(\omega))=(-1)^n[\mathcal{Z}]^2\mathscr{H}_{\mathcal{V}+\sigma}(\zeta(\omega))$.
     \end{description}
 \end{theo}
 \begin{proof}
    (i)  By using the definition of the operator $\mathcal{N}_{\mathcal{V}}$ and $n$-dimensional bicomplex Hankel transformation \eqref{eq:h11}, we get
    \begin{align}\label{eq:s21}
       &\mathscr{H}_{\mathcal{V}+\sigma+1}(\mathcal{N}_{\mathcal{V}+\sigma}\zeta(\omega))\nonumber\\&=\int_{0}^{\infty}\ldots\int_{0}^{\infty}[\omega]^{\mathcal{V}+\sigma+\frac{1}{2}}\frac{\partial^n}{\partial\omega_1...\partial\omega_n}[\omega]^{-\mathcal{V}-\sigma-\frac{1}{2}}\zeta(\omega) \prod_{k=1}^{n}\sqrt{\omega_kZ_k}\;\mathcal{J}_{\mathcal{V}+\sigma+1}\left(\omega_kZ_k\right) d\omega_1\ldots d\omega_n \nonumber\\
       &=[\mathcal{Z}]^\frac{1}{2}\int_{0}^{\infty}\ldots\int_{0}^{\infty}[\omega]^{\mathcal{V}+\sigma+1}\prod_{k=1}^{n}\mathcal{J}_{\mathcal{V}+\sigma+1}\left(\omega_kZ_k\right)\frac{\partial^n}{\partial\omega_1...\partial\omega_n}[\omega]^{-\mathcal{V}-\sigma-\frac{1}{2}}\zeta(\omega)  d\omega_1\ldots d\omega_n
    \end{align}
    Now we applying integration by parts and using the result
    \begin{align*}
        \frac{\partial}{\partial \omega_1}\omega_1^{\mathcal{V}+\sigma+1}\mathcal{J}_{\mathcal{V}+\sigma+1}(\omega_1Z_1)=\omega_1^{\mathcal{V}+\sigma+1}Z_1\mathcal{J}_{\mathcal{V}+\sigma}(\omega_1Z_1)
    \end{align*}
    in \eqref{eq:s21}, we have
    \begin{align}\label{eq:f123}
        &\mathscr{H}_{\mathcal{V}+\sigma+1}(\mathcal{N}_{\mathcal{V}+\sigma}\zeta(\omega))\nonumber\\&=\int_{0}^{\infty}\ldots\int_{0}^{\infty}\left\{\left[\omega_1^{\mathcal{V}+\sigma+1}\mathcal{J}_{\mathcal{V}+\sigma+1}\frac{\partial^n}{\partial\omega_2...\partial\omega_n}[\omega]^{-\mathcal{V}-\sigma-\frac{1}{2}}\zeta(\omega)\right]_0^\infty\right.\nonumber\\&\left.\hspace{80pt}-\int_{0}^\infty \omega_1^{\mathcal{V}+\sigma+1}Z_1\mathcal{J}_{\mathcal{V}+\sigma}(\omega_1Z_1)\frac{\partial^n}{\partial\omega_2...\partial\omega_n}[\omega]^{-\mathcal{V}-\sigma-\frac{1}{2}}\zeta(\omega)\right\}\nonumber\\&\hspace{60pt}\times[\mathcal{Z}]^\frac{1}{2}\prod_{k=2}^{n}\omega_k^{\mathcal{V}+\sigma+1}\mathcal{J}_{\mathcal{V}+\sigma+1}\left(\omega_kZ_k\right)  d\omega_2\ldots d\omega_n.
    \end{align}
    Since $\zeta(\omega)=0$ on $E_1$, $\sigma$ is a positive hyperbolic number and $\omega_1^{\mathcal{V}+\sigma+1}\mathcal{J}_{\mathcal{V}+\sigma+1}(\omega_1Z_1)=O(\omega_1^{2\sigma+1})$ as $\omega_1\rightarrow0$, then the limit terms of \eqref{eq:f123} vanish. Similarly by using integration by parts through the subsequent components $\omega_2,...,\omega_n$, we obtain
    \begin{align*}
        \mathscr{H}_{\mathcal{V}+\sigma+1}(\mathcal{N}_{\mathcal{V}+\sigma}\zeta(\omega))&=(-1)^n\int_{0}^{\infty}\ldots\int_{0}^{\infty}\prod_{k=1}^{n}Z_k(\omega_k)^\frac{1}{2}\mathcal{J}_{\mathcal{V}+\sigma}(\omega_1Z_1)\zeta(\omega)[\mathcal{Z}]^{\frac{1}{2}}d\omega_1\ldots d\omega_n\\
        &=(-1)^n[\mathcal{Z}]\mathscr{H}_{\mathcal{V}+\sigma}(\zeta(\omega)).
    \end{align*}
    Hence, the proof of (i) is complete. The proof of (ii) is similar to the proof of (i). By using the result of (i) and (ii) we can easily proof (iii).
 \end{proof}
  The $n$-dimensional bicomplex Hankel transform's inversion formula is now defined as follows:
  \begin{theo}
  Let $\mathcal{V}=\alpha_1+j\alpha_2=\nu_1e_1+\nu_2e_2\in\mathbb{BC},\zeta(\omega)\in \mathcal{A}_{\sigma,\varepsilon},\mathscr{H}_{\mathcal{V}}(\zeta(\omega))=\eta(\mathcal{Z})$ and we restricted $\mathcal{Z}$ on $\{\mathcal{Z}=(Z_1,...,Z_n):Z_k\in\mathbb{R},k=1,2,...,n\}$ then for $\Re(\alpha_1)+\frac{1}{2}\geq|\Im(\alpha_2)|$, $\mathscr{H}_{\mathcal{V}}^{-1}\eta(\mathcal{Z})$ is given by
  \begin{align}\label{eq:a156}
\lim_{s\rightarrow\infty}\int_{0}^{s}\ldots\int_{0}^{s}\eta(\mathcal{Z})\prod_{k=1}^{n}\sqrt{\omega_kZ_k}\;\mathcal{J}_\mathcal{V}\left(\omega_kZ_k\right) dZ_1\ldots dZ_n.
  \end{align}
  \end{theo}
 \begin{proof}
     Since, $\mathscr{H}_{\mathcal{V}}(\zeta(\omega))=\eta(\mathcal{Z})$ then by using \eqref{eq:h11}, we obtain
     \begin{align}\label{eq:6j}
         &\lim_{s\rightarrow\infty}\int_{0}^{s}\ldots\int_{0}^{s}\eta(\mathcal{Z})\prod_{k=1}^{n}\sqrt{\omega_kZ_k}\;\mathcal{J}_\mathcal{V}\left(\omega_kZ_k\right) dZ_1\ldots dZ_n\nonumber\\
         &=\lim_{s\rightarrow\infty}\int_{0}^{s}\ldots\int_{0}^{s}\left[\int_{0}^{\infty}\ldots\int_{0}^{\infty}\zeta(t) \prod_{k=1}^{n}\sqrt{t_kZ_k}\;\mathcal{J}_\mathcal{V}\left(t_kZ_k\right) dt_1\ldots dt_n\right]\nonumber\\&\hspace{150pt}\times\prod_{k=1}^{n}\sqrt{\omega_kZ_k}\;\mathcal{J}_\mathcal{V}\left(\omega_kZ_k\right) dZ_1\ldots dZ_n\nonumber\\
         &=\lim_{s\rightarrow\infty}\int_{0}^{s}\ldots\int_{0}^{s}\int_{0}^{\infty}\ldots\int_{0}^{\infty}\zeta(t) \prod_{k=1}^{n}\sqrt{t_k\omega_k}Z_k\;\mathcal{J}_\mathcal{V}\left(t_kZ_k\right) \mathcal{J}_\mathcal{V}\left(\omega_kZ_k\right) dt_1\ldots dt_n dZ_1\ldots dZ_n\nonumber\\
         &=\lim_{s\rightarrow\infty}\int_{0}^{s}\ldots\int_{0}^{s}\int_{0}^{\infty}\ldots\int_{0}^{\infty}\zeta_1(t) \prod_{k=1}^{n}\sqrt{t_k\omega_k}Z_k\;J_{\nu_1}\left(t_kZ_k\right) J_{\nu_1}\left(\omega_kZ_k\right) dt_1\ldots dt_n dZ_1\ldots dZ_ne_1\nonumber\\
         &+\lim_{s\rightarrow\infty}\int_{0}^{s}\ldots\int_{0}^{s}\int_{0}^{\infty}\ldots\int_{0}^{\infty}\zeta_2(t) \prod_{k=1}^{n}\sqrt{t_k\omega_k}Z_k\;J_{\nu_2}\left(t_kZ_k\right) J_{\nu_2}\left(\omega_kZ_k\right) dt_1\ldots dt_n dZ_1\ldots dZ_ne_2.
     \end{align}
   Applying the inversion formula of Hankel transform \cite{transforms}
     $$f(x)=\int_{0}^{\infty} \sqrt{yx} J_\nu\left(yx\right)\int_{0}^{\infty} \sqrt{yt} J_\nu\left(yt\right)f(t)dtdy$$ for $\nu\geq-\frac{1}{2}$ in the first integral of \eqref{eq:6j}, we obtain
     \begin{align*}
     &\lim_{s\rightarrow\infty}\int_{0}^{s}\ldots\int_{0}^{s}\int_{0}^{\infty}\ldots\int_{0}^{\infty}\zeta_1(t) \prod_{k=1}^{n}\sqrt{t_k\omega_k}Z_k\;J_{\nu_1}\left(t_kZ_k\right) J_{\nu_1}\left(\omega_kZ_k\right) dt_1\ldots dt_n dZ_1\ldots dZ_n\\
     &=\lim_{s\rightarrow\infty}\int_{0}^{s}\ldots\int_{0}^{s}\int_{0}^{\infty}\ldots\int_{0}^{\infty} \prod_{k=1}^{n-1}\sqrt{t_k\omega_k}Z_k\;J_{\nu_1}\left(t_kZ_k\right) J_{\nu_1}\left(\omega_kZ_k\right)\\&\hspace{55pt}\times\left[\int_{0}^{s}\int_{0}^{\infty}\zeta_1(t)\sqrt{t_n\omega_n}Z_nJ_{\nu_1}\left(t_nZ_n\right) J_{\nu_1}\left(\omega_nZ_n\right)dt_ndZ_n\right] dt_1\ldots dt_{n-1} dZ_1\ldots dZ_{n-1}\\
     &=\lim_{s\rightarrow\infty}\int_{0}^{s}\ldots\int_{0}^{s}\int_{0}^{\infty}\ldots\int_{0}^{\infty} \prod_{k=1}^{n-1}\sqrt{t_k\omega_k}Z_k\;J_{\nu_1}\left(t_kZ_k\right) J_{\nu_1}\left(\omega_kZ_k\right)\\&\hspace{200pt}\times\zeta_1(t_1,...,t_{n-1},\omega_n) dt_1\ldots dt_{n-1} dZ_1\ldots dZ_{n-1}.
     \end{align*}
     Using the similar procedure, after solving component wise, we obtain
     \begin{align}\label{eq:i3}
         \lim_{s\rightarrow\infty}\int_{0}^{s}\ldots\int_{0}^{s}\int_{0}^{\infty}\ldots\int_{0}^{\infty}\zeta_1(t) \prod_{k=1}^{n}\sqrt{t_k\omega_k}Z_k\;J_{\nu_1}\left(t_kZ_k\right) J_{\nu_1}\left(\omega_kZ_k\right) dt_1\ldots dt_n dZ_1\ldots dZ_n=\zeta_1(\omega),
     \end{align}
     for $\nu_1\geq-\frac{1}{2}$. Similarly,
     \begin{align}\label{eq:i6}
         \lim_{s\rightarrow\infty}\int_{0}^{s}\ldots\int_{0}^{s}\int_{0}^{\infty}\ldots\int_{0}^{\infty}\zeta_2(t) \prod_{k=1}^{n}\sqrt{t_k\omega_k}Z_k\;J_{\nu_2}\left(t_kZ_k\right) J_{\nu_2}\left(\omega_kZ_k\right) dt_1\ldots dt_n dZ_1\ldots dZ_n=\zeta_2(\omega)
     \end{align}
     for $\nu_2\geq-\frac{1}{2}$. Using \eqref{eq:i3}, \eqref{eq:i6} and \eqref{eq:6j}, we have
     \begin{align*}
         \lim_{s\rightarrow\infty}\int_{0}^{s}\ldots\int_{0}^{s}\eta(\mathcal{Z})\prod_{k=1}^{n}\sqrt{\omega_kZ_k}\;\mathcal{J}_\mathcal{V}\left(\omega_kZ_k\right) dZ_1\ldots dZ_n=\zeta_1(\omega)e_1+\zeta_2(\omega)e_2=\zeta(\omega).
     \end{align*}
     Thus, the proof of the theorem is complete.
 \end{proof}
 \begin{theo}
    Assume that $\mathcal{V}=\alpha_1+j\alpha_2=\nu_1e_1+\nu_2e_2\in\mathbb{BC}$, where $e_1=\frac{1+k}{2}, e_2=\frac{1-k}{2}$ and satisfies the inequality $\Re(\alpha_1)+\frac{1}{2}\geq|\Im(\alpha_2)|$, then $\mathscr{H}_{\mathcal{V}}$ is an isomorphism from $\mathcal{A}_{\sigma,\varepsilon}$ to  $\mathcal{C}^\mathcal{V}_\rho$.
 \end{theo}
 \begin{proof}
     Let us consider $\zeta(\omega)\in\mathcal{A}_{\mathcal{V},\varepsilon}$. Then \begin{align}\label{eq:x9}
     \mathscr{H}_{\mathcal{V}}(\zeta(\omega))=\eta(\mathcal{Z})=\int_{0}^{\infty}\ldots\int_{0}^{\infty}\zeta(\omega) \prod_{k=1}^{n}\sqrt{\omega_kZ_k}\;\mathcal{J}_\mathcal{V}\left(\omega_kZ_k\right) d\omega_1\ldots d\omega_n.\end{align}
     Using (i) of Theorem \ref{th:f3}, we obtain
     \begin{align}\label{eq:7g}
         \mathscr{H}_{\mathcal{V}+2m}(\mathcal{N}_{\mathcal{V}+2m-1}\ldots\mathcal{N}_{\mathcal{V}}\zeta(\omega))=(-1)^n[\mathcal{Z}]^{2m} \mathscr{H}_\mathcal{V}(\zeta(\omega))
     \end{align}
     Now using \eqref{eq:x9} and \eqref{th:f3}, we have
     \begin{align}\label{eq:n71}
       &(-1)^n\exp\left(-\rho\sum_{k=1}^{n}|Z_k|_h\right)[\mathcal{Z}]^{2m-\mathcal{V}-\frac{1}{2}} \eta(\mathcal{Z})\nonumber\\
       &=\exp\left(-\rho\sum_{k=1}^{n}|Z_k|_h\right)[\mathcal{Z}]^{-\mathcal{V}-\frac{1}{2}} \mathscr{H}_{\mathcal{V}+2m}(\mathcal{N}_{\mathcal{V}+2m-1}\ldots\mathcal{N}_{\mathcal{V}}\zeta(\omega))\nonumber\\
       &=\exp\left(-\rho\sum_{k=1}^{n}|Z_k|_h\right)\int_{0}^{\infty}\ldots\int_{0}^{\infty}[\mathcal{N}_{\mathcal{V}+2m-1}\ldots\mathcal{N}_{\mathcal{V}}\zeta(\omega)] \prod_{k=1}^{n}\omega_k^{\frac{1}{2}}Z_k^{-\mathcal{V}}\;\mathcal{J}_{\mathcal{V}+2m}\left(\omega_kZ_k\right) d\omega_1\ldots d\omega_n\nonumber\\
       &=\int_{0}^{\infty}\ldots\int_{0}^{\infty}\left\{[\omega]^{\mathcal{V}+2m+\frac{1}{2}}(\omega^{-1}D_\omega)^{2m}[\omega]^{-\mathcal{V}-\frac{1}{2}}\zeta(\omega)\right\}\nonumber\\&\hspace{75pt}\times \prod_{k=1}^{n}\omega_k^{\frac{1}{2}}Z_k^{-\mathcal{V}}\exp\left(-\rho|Z_k|_h\right)\mathcal{J}_{\mathcal{V}+2m}\left(\omega_kZ_k\right) d\omega_1\ldots d\omega_n\nonumber \\
       &=\int_{0}^{\infty}\ldots\int_{0}^{\infty}\left\{[\omega]^{2\mathcal{V}+2m+1}(\omega^{-1}D_\omega)^{2m}[\omega]^{-\mathcal{V}-\frac{1}{2}}\zeta(\omega)\right\}\nonumber\\&\hspace{75pt}\times \prod_{k=1}^{n}(\omega_kZ_k)^{-\mathcal{V}}\exp\left(-\rho|Z_k|_h\right)\mathcal{J}_{\mathcal{V}+2m}\left(\omega_kZ_k\right) d\omega_1\ldots d\omega_n.
     \end{align}
     Using \eqref{eq:c96} and \eqref{eq:n71}, we get
     \begin{align*}
         \left|\exp\left(-\rho\sum_{k=1}^{n}|Z_k|_h\right)[\mathcal{Z}]^{2m-\mathcal{V}-\frac{1}{2}} \eta(\mathcal{Z})\right|_h&<_h\mathcal{T}^{\mathcal{V},\varepsilon}_{2m}(\zeta(\omega))\prod_{k=1}^{n}A_k\int_{0}^{\varepsilon}\ldots\int_{0}^{\varepsilon}[\omega]^{2\mathcal{V}+2m+1}d\omega_1\ldots d\omega_n\\
         &=A\mathcal{T}^{\mathcal{V},\varepsilon}_{2m}(\zeta(\omega)),
  \end{align*}
  where $A$ is a constant. Therefore, $\mathcal{G}^\mathcal{V}_{\rho,m}(\eta(\mathcal{Z}))<_hA\mathcal{T}^{\mathcal{V},\varepsilon}_{2m}(\zeta(\omega))$. Thus, $\mathscr{H}_\mathcal{V}$ is a linear continuous mapping from $\mathcal{A}_{\mathcal{V},\varepsilon}$ to  $\mathcal{C}^\mathcal{V}_\rho$. Conversely let, $\eta(\mathcal{Z})\in\mathcal{C}^\mathcal{V}_\rho$ and $\zeta(\omega)=\mathscr{H}_\mathcal{V}^{-1}(\eta(\mathcal{Z}))$. Now using inversion formula of bicomplex Hankel transform \eqref{eq:a156}, we obtain
  \begin{align*}
\zeta(\omega)&=\lim_{s\rightarrow\infty}\int_{0}^{s}\ldots\int_{0}^{s}\eta(\mathcal{Z})\prod_{k=1}^{n}\sqrt{\omega_kZ_k}\;\mathcal{J}_\mathcal{V}\left(\omega_kZ_k\right) dZ_1\ldots dZ_n
  \end{align*}
  By using differentiating under the integral sign through each variable $\omega_k$ for $k=1,2,...,n$ , we get
  \begin{align}\label{eq:g78}
      &(\omega^{-1}D_\omega)^p[\omega]^{-\mathcal{V}-\frac{1}{2}}\zeta(\omega)\nonumber\\
      &=\lim_{s\rightarrow\infty}\int_{0}^{s}\ldots\int_{0}^{s}(-1)^n\eta(\mathcal{Z})[\mathcal{Z}]^{2p+\mathcal{V}+\frac{1}{2}}\prod_{k=1}^{n}(\omega_kZ_k)^{-\mathcal{V}-p}\mathcal{J}_{\mathcal{V}+p}\left(\omega_kZ_k\right) dZ_1\ldots dZ_n
  \end{align}
  Since, $\eta(\mathcal{Z})$ is of rapid descent for and $(\omega_kz_{ik})^{-\nu_i-p}J_{\nu_i+p}\left(\omega_kz_{ik}\right)$ is bounded on $0<\omega_kz_{ik}<\infty$ for $i=1,2$ then
  \begin{align}\label{eq:d43}
      \prod_{k=1}^{n}(\omega_kZ_k)^{-\mathcal{V}}\mathcal{J}_{\mathcal{V}+p}\left(\omega_kZ_k\right)&=\prod_{k=1}^{n}(\omega_kz_{1k})^{-\nu_1}J_{\nu_1+p}\left(\omega_kz_{1k}\right)e_1+\prod_{k=1}^{n}(\omega_kz_{2k})^{-\nu_2}J_{\nu_2+p}\left(\omega_kz_{2k}\right)e_2\nonumber\\
       &<_hB,\;\;0<|\omega_kZ_k|_h<\infty,
  \end{align}
  where $B$ is a constant.
  Using \eqref{eq:g78} and \eqref{eq:d43}, we have
  \begin{align}
      \mathcal{T}^{\sigma,\varepsilon}_{p,p}(\zeta(\omega))&=\sup\left|(\omega^{-1}D_\omega)^p[\omega]^{-\mathcal{V}-\frac{1}{2}}\zeta(\omega)\right|_h\nonumber\\
      &<_hB\lim_{s\rightarrow\infty}\int_{0}^{s}\ldots\int_{0}^{s}\sup\left|\eta(\mathcal{Z})\prod_{k=1}^{n}\left(1+Z_k^2\right)[\mathcal{Z}]^{2p+\mathcal{V}+\frac{1}{2}}\right|_h\prod_{k=1}^{n}\frac{1}{(1+Z_k^2)} dZ_1\ldots dZ_n\nonumber\\
      &=B\left(\frac{\pi}{2}\right)^n\sup\left|\prod_{k=1}^{n}\left(1+Z_k^2\right)[\mathcal{Z}]^{2p+\mathcal{V}+\frac{1}{2}}\eta(\mathcal{Z})\right|_h\nonumber\\
      &=B\left(\frac{\pi}{2}\right)^n\sup\left|\prod_{k=1}^{n}\left(1+Z_k^2\right)[\mathcal{Z}]^{2p+2\mathcal{V}+1}[\mathcal{Z}]^{-\mathcal{V}-\frac{1}{2}}\eta(\mathcal{Z})\right|_h.
  \end{align}
   Again, if $q$ is an integer such that $\mathcal{V}+p+\frac{1}{2}\leq_h q$, then
  \begin{align}\label{eq:f45}
      [\mathcal{Z}]^{2p+2\mathcal{V}+1}=\prod_{k=1}^{n}Z_k^{2\left(\mathcal{V}+p+\frac{1}{2}\right)}<_h\prod_{k=1}^{n}\left(1+Z_k^2\right)^q
  \end{align}
  Therefore,\begin{align*}
      \mathcal{T}^{\sigma,\varepsilon}_{p,p}(\zeta(\omega))&<_hB\left(\frac{\pi}{2}\right)^n\sup\left|\prod_{k=1}^{n}\left(1+Z_k^2\right)^{q+1}[\mathcal{Z}]^{-\mathcal{V}-\frac{1}{2}}\eta(\mathcal{Z})\right|_h\\
      &<_hB\left(\frac{\pi}{2}\right)^n\sup\left|\exp\left(-\rho\sum_{k=1}^{n}|Z_k|_h\right)\prod_{k=1}^{n}\left(1+Z_k^2\right)^{q+1}[\mathcal{Z}]^{-\mathcal{V}-\frac{1}{2}}\eta(\mathcal{Z})\right|_h\\
      &=B\left(\frac{\pi}{2}\right)^n\sum_{r=0}^{q+1}{}^{q+1}C_r\;\mathcal{G}^\mathcal{V}_{\rho,r}(\eta(\mathcal{Z})).
       \end{align*}
       Hence, $\mathscr{H}_\mathcal{V}^{-1}$ is a continuous linear mapping $\mathcal{C}^\mathcal{V}_\rho$ to $\mathcal{A}_{\mathcal{V},\varepsilon}$. Since, $\mathscr{H}_\mathcal{V}$ is linear and homeomorphism therefore it is an isomorphism. Hence the proof is completed.
 \end{proof}

 Now we shall illustrate, through specific examples, the advantages of employing the $n$-dimensional Hankel transformation to solve partial differential equations that involve the operator $\mathcal{M}_{\mathcal{V}}\mathcal{N}_{\mathcal{V}}$.
 \begin{example}\label{ex1}
     Let us consider the partial differential equation
     \begin{align}\label{eq:dp}
         \mathcal{M}_{\mathcal{V}}\mathcal{N}_{\mathcal{V}}u(\omega_1,...,\omega_n,t)=\lambda^2\frac{\partial^2}{\partial t^2}u(\omega_1,...,\omega_n,t),
     \end{align}
     where $u(\omega_1,...,\omega_n,t)\in \mathcal{A}_{\mathcal{V},\varepsilon}$, with satisfies the initial conditions $u(x_1,...,x_n,t)\rightarrow f(x_1,...,x_n)$ as $t\rightarrow0$ and $\frac{\partial u}{\partial t}\rightarrow g(\omega_1,...,\omega_n)$ as $t\rightarrow 0$.
 \end{example}
  By taking both sides $n$-dimensional bicomplex Hankel transform $\mathscr{H}_\mathcal{V}$ on
\eqref{eq:dp}, we get
\begin{align}\label{eq:x7}
    (-1)^{n}[\mathcal{Z}]^2U(Z_1,...,Z_n,t)=\lambda^2\frac{\partial^2}{\partial t^2}U(Z_1,...,Z_n,t),
\end{align}
with $U(Z_1,...,Z_n,0)=F(Z_1,...,Z_n)=\mathscr{H}_\mathcal{V}(f(\omega_1,...,\omega_n))$ and $\frac{\partial U(Z_1,...Z_n,0)}{\partial t}=G(Z_1,...,Z_n)$. If $n$ is even positive integer then solution of the equation \eqref{eq:x7} is given by
\begin{align*}
    U(Z_1,...,Z_n,t)&=\frac{F(Z_1,...,Z_n)[\mathcal{Z}]+\lambda G(Z_1,...,Z_n)}{2[\mathcal{Z}]}\exp\left(\frac{[\mathcal{Z}]t}{\lambda}\right)\\&+\frac{F(Z_1,...,Z_n)[\mathcal{Z}]-\lambda G(Z_1,...,Z_n)}{2[\mathcal{Z}]}\exp\left(-\frac{[\mathcal{Z}]t}{\lambda}\right).
\end{align*}
For odd positive integer value of $n$, we get
\begin{align*}
     U(Z_1,...,Z_n,t)=F(Z_1,...,Z_n)\cos\left(\frac{[\mathcal{Z}]t}{\lambda}\right)+G(Z_1,...,Z_n)\sin\left(\frac{[\mathcal{Z}]t}{\lambda}\right).
\end{align*}
 Now taking the inverse of the $n$-dimensional bicomplex Hankel transform we get the results
 \begin{align*}
 u(\omega_1,...,\omega_n,t)=\lim_{s\rightarrow\infty}\int_{0}^{s}\ldots\int_{0}^{s} U(Z_1,...,Z_n,t)\prod_{k=1}^{n}\sqrt{\omega_kZ_k}\;\mathcal{J}_\mathcal{V}\left(\omega_kZ_k\right) dZ_1\ldots dZ_n.
 \end{align*}
 \begin{remark}\label{re2}
     Putting $n=1$ in \eqref{eq:dp} we obtain bicomplex generalized wave equation
     \begin{align}\label{eq:c59}
        \frac{\partial^2}{\partial\omega_1^2}u(\omega_1,t)-\frac{4\mathcal{V}^2-1}{4\omega_1^2}u(\omega_1,t)=\lambda^2\frac{\partial^2}{\partial t^2}u(\omega_1,t),
    \end{align}
    whose solution is given by
    \begin{align*}
        u(\omega_1,t)=\lim_{s\rightarrow\infty}\int_{0}^{s}\left[F(Z_1)\cos\left(\frac{Z_1t}{\lambda}\right)+G(Z_1)\sin\left(\frac{Z_1t}{\lambda}\right)\right]\sqrt{\omega_1Z_k}\mathcal{J}_\mathcal{V}\left(\omega_1Z_1\right) dZ_1,
    \end{align*}
    where $F(Z_1)=\mathscr{H}_\mathcal{V}(f(\omega_1))$ and $G(Z_1)=\mathscr{H}_\mathcal{V}(g(\omega_1))$.
 \end{remark}
 \begin{remark}
    Setting $\mathcal{V}=-\frac{1}{2}$ in \eqref{eq:c59}, we obtain two dimensional classical wave equation
    \begin{align*}
        \frac{\partial^2}{\partial\omega_1^2}u(\omega_1,t)=\lambda^2\frac{\partial^2}{\partial t^2}u(\omega_1,t),
    \end{align*}
    and the solution is given by
    \begin{align}\label{eq:x96}
        u(\omega_1,t)=\lim_{s\rightarrow\infty}\int_{0}^{s}\left[F(Z_1)\cos\left(\frac{Z_1t}{\lambda}\right)+G(Z_1)\sin\left(\frac{Z_1t}{\lambda}\right)\right]\sqrt{\omega_1Z_1}\mathcal{J}_{-\frac{1}{2}}\left(\omega_1Z_1\right) dZ_1,
    \end{align}
    where $F(Z_1)=\mathscr{H}_\mathcal{V}(f(\omega_1))$ and $G(Z_1)=\mathscr{H}_\mathcal{V}(g(\omega_1))$.
 \end{remark}
    \begin{remark}
       Putting  \begin{equation*}
        f(\omega_1) =
        \begin{cases}
        \displaystyle 1, & \text{if } 0<\omega_1<1, \\
            \\
            0, & \text{otherwise},
        \end{cases}
       ~~~~~ \text{and}\;\;\;\; g(\omega_1)=0
    \end{equation*}
    in \eqref{eq:x96} of Remark 3 we obtain the solution of wave equation as
    \begin{align}\label{eq:d92}
        u(\omega_1,t)=\lim_{s\rightarrow\infty}\int_{0}^{s}\left[\int_0^1\sqrt{\omega_1Z_1}\mathcal{J}_{-\frac{1}{2}}\left(\omega_1Z_1\right)d\omega_1\right]\cos\left(\frac{Z_1t}{\lambda}\right)\sqrt{\omega_1Z_k}\mathcal{J}_{-\frac{1}{2}}\left(\omega_1Z_1\right) dZ_1,
    \end{align}
    and graphical representation we have plotted in Figure~\ref{fig:wave}.
    \end{remark}

\begin{figure}[ht!]
  \centering
    \includegraphics[width=0.6\linewidth]{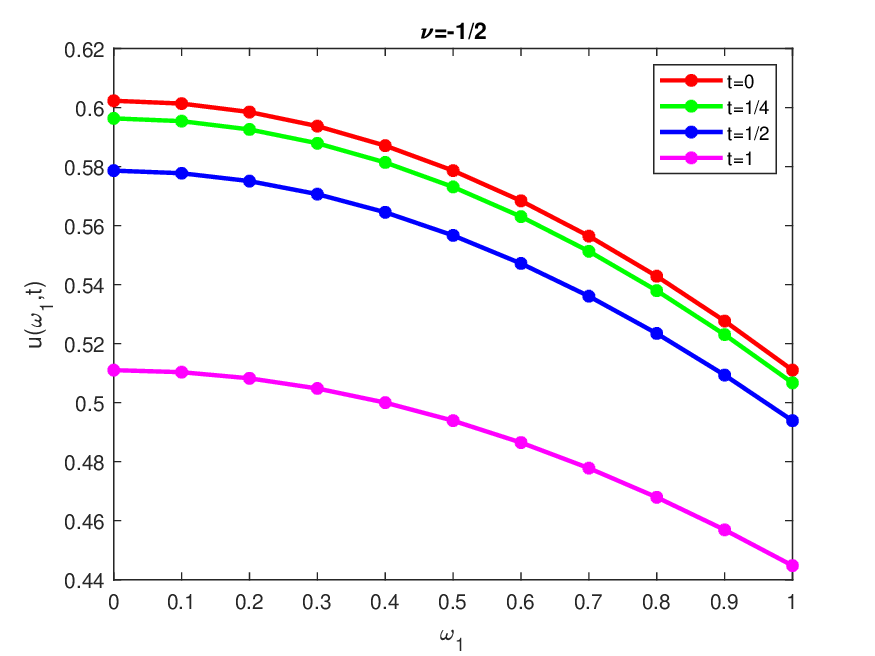}
    \caption{ $u(\omega_1,t)$ for $\omega_1\in[0,1]$, $\lambda=1$ in equation \eqref{eq:d92}.}
  \label{fig:wave}
\end{figure}

\begin{example}\label{ex5}
     We investigate the solution of another partial differential equation
     \begin{align}\label{eq:c1}
\mathcal{M}_{\mathcal{V}}\mathcal{N}_{\mathcal{V}}u(\omega_1,...,\omega_n,t)=\lambda\frac{\partial}{\partial t}u(\omega_1,...,\omega_n,t),
     \end{align}
     where $u(\omega_1,...,\omega_n,t)\in \mathcal{A}_{\mathcal{V},\varepsilon}$, with satisfies the initial conditions $u(\omega_1,...,\omega_n,t)\rightarrow f(\omega_1,...,\omega_n)$ as $t\rightarrow0$.
 \end{example}
 Similarly, after applying $n$-dimensional bicomplex Hankel transform both side on \eqref{eq:c1}, we have
 \begin{align*}
     \lambda\frac{\partial}{\partial t}U(Z_1,...,Z_n,t)=(-1)^{n}[\mathcal{Z}]^2U(Z_1,...,Z_n,t).
 \end{align*}
 Then the corresponding solution is given by
 \begin{align*}
     U(Z_1,...,Z_n,t)=F(Z_1,...,Z_n)\exp\left(\frac{(-1)^{n}[\mathcal{Z}]^2}{\lambda}\right).
 \end{align*}
 By using inversion formula of $n$-dimensional bicomplex Hankel transform, we obtain
 \begin{align*}
     u(\omega_1,...,\omega_n,t)=\lim_{s\rightarrow\infty}\int_{0}^{s}\ldots\int_{0}^{s} F(Z_1,...,Z_n)\exp\left(\frac{(-1)^{n}[\mathcal{Z}]^2}{\lambda}\right)\prod_{k=1}^{n}\sqrt{\omega_kZ_k}\;\mathcal{J}_\mathcal{V}\left(\omega_kZ_k\right) dZ_1\ldots dZ_n.
 \end{align*}
 \begin{remark}\label{re6}
     Setting $n=1$ in \eqref{eq:c1}, we get bicomplex generalized heat equation
     \begin{align}\label{eq:x53}
         \frac{\partial^2}{\partial\omega_1^2}u(\omega_1,t)-\frac{4\mathcal{V}^2-1}{4\omega_1^2}u(\omega_1,t)=\lambda\frac{\partial}{\partial t}u(\omega_1,t),
     \end{align}
     whose solution is given by
     \begin{align*}
     u(\omega_1,t)=\lim_{s\rightarrow\infty}\int_{0}^{s} F(Z_1)\exp\left(\frac{- Z_1^2}{\lambda}\right)\sqrt{\omega_1Z_1}\;\mathcal{J}_\mathcal{V}\left(\omega_1Z_1\right) dZ_1,
 \end{align*}
 where $F(Z_1)=\mathscr{H}_\mathcal{V}(f(\omega_1))$.
 \end{remark}
 \begin{remark}
     For $\mathcal{V}=-\frac{1}{2}$ the equation \eqref{eq:x53} reduces to heat equation
     \begin{align*}
         \frac{\partial^2}{\partial\omega_1^2}u(\omega_1,t)=\lambda\frac{\partial}{\partial t}u(\omega_1,t),
     \end{align*}
     and solution is given by
     \begin{align}\label{eq:i73}
         u(\omega_1,t)=\lim_{s\rightarrow\infty}\int_{0}^{s} F(Z_1)\exp\left(-\frac{Z_1^2}{\lambda}\right)\sqrt{\omega_1Z_1}\;\mathcal{J}_{-\frac{1}{2}}\left(\omega_1Z_1\right) dZ_1,
     \end{align}
     where $F(Z_1)=\mathscr{H}_\mathcal{V}(f(\omega_1))$.
 \end{remark}
 \begin{remark}
     Putting  \begin{equation*}
        f(\omega_1) =
        \begin{cases}
        \displaystyle 1, & \text{if } 0<\omega_1<1, \\
            \\
            0, & \text{otherwise},
        \end{cases}
    \end{equation*}
    in \eqref{eq:i73}, we have the solution of heat equation as
    \begin{align}\label{eq:x83}
         u(\omega_1,t)=\lim_{s\rightarrow\infty}\int_{0}^{s} \left[\int_0^1\sqrt{\omega_1Z_1}\mathcal{J}_{-\frac{1}{2}}\left(\omega_1Z_1\right)d\omega_1\right]\exp\left(-\frac{Z_1^2}{\lambda}\right)\sqrt{\omega_1Z_1}\;\mathcal{J}_{-\frac{1}{2}}\left(\omega_1Z_1\right) dZ_1,
    \end{align}
    and graphical representation we have plotted in Figure~\ref{fig:heat}.
 \end{remark}
 \begin{figure}[ht!]
  \centering
    \includegraphics[width=0.6\linewidth]{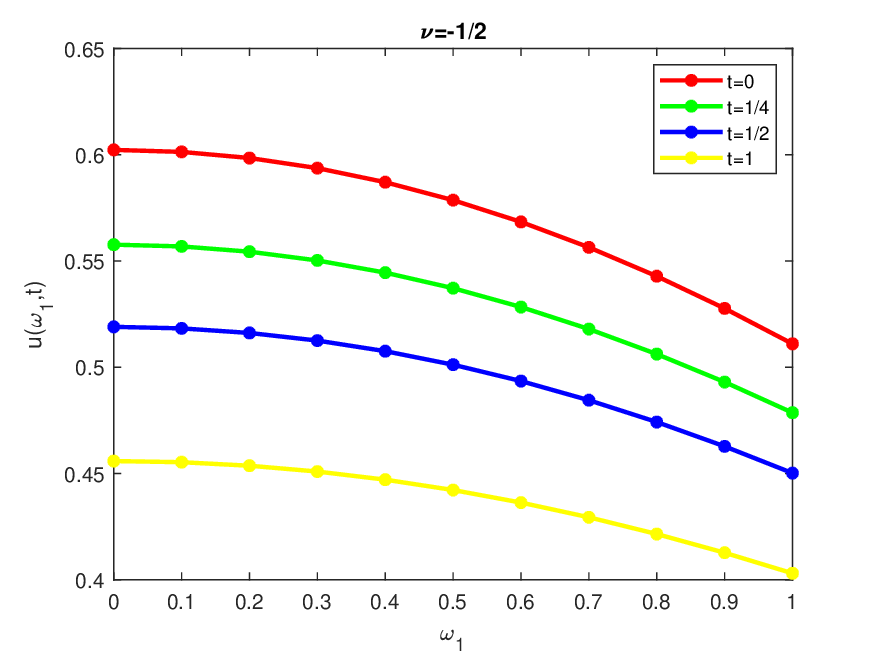}
    \caption{$u(\omega_1,t)$ for $\omega_1\in[0,1]$, $\lambda=1$ in equation \eqref{eq:x83}.}
  \label{fig:heat}
\end{figure}
\section{ Bicomplex generalized coherent states associated with bicomplex Bessel function}
In recent decades, coherent states have attracted considerable scientific interest because of their wide application in fields such as condensed matter physics, mathematical physics, signal processing and quantum information. Various extension of coherent states have been developed, including Mittag Lafler coherent states \cite{ml_cs}, generalized hypergeometric coherent states \cite{gh_cs}, coherent states for generalized Laguerre functions \cite{lg_cs} etc. In \cite{int_mittage}, the author constructed and studied the characteristics of a generalized integral multi-index Mittag-Leffler function and also developed and analyzed the properties of the coherent states related to this function for continuous spectrum. Numerous important books and research articles have been written on coherent states and their various applications \cite{bc_hamiltonian,coherent_app,deformed_osc,coh_incoh,gen_coherent_app}. In earlier work \cite{bicomplex ghf}, Fock states were defined within an infinite-dimensional bicomplex Hilbert space as follows:
\begin{equation}\label{eq:fost}
\mid n>=\sum_{l=1}^{2}\mid n_l> e_l,\quad n_l=0,1,2,...\nonumber
\end{equation}
and also we discus the orthogonality and completeness conditions of the Fock states which are presented as follows:
\begin{equation}\label{csn3}
<n\mid m>=\sum_{l=1}^{2}\delta_{n_l m_l}e_l\;\;\mbox{and}\;\;\sum_{n=0}^{\infty}\mid n><n\mid=1.
\end{equation}
In this section, we define bicomplex generalized states related to bicomplex Bessel function  $\mathcal{J}_\mathcal{V}(Z)$. Moreover, we prove those states are coherent states and satisfies the following properties:
\begin{itemize}
    \item [(a)] Normalization: $<Z \mid Z>=1$,
    \item [(b)] Continuity: $|Z'-Z|_h\rightarrow 0$ implies that $||\mid Z'>-\mid Z>||\rightarrow 0$,
    \item[(c)] Resolution of unity: A weight function $\mathcal{W}(|Z|_h)>_h0$, is chosen such that the integral $\int d \mu(Z)\mid Z><Z\mid=1$ holds, where the integration measure is given by $$d\mu(Z)=\frac{d\psi_Z}{2\pi}d(|Z|^2_h)\mathcal{W}(|Z|_h)$$ and hyperbolic angle is expressed as $\psi_Z=\theta_1e_1+\theta_2e_2,\;0\leq\theta_1,\theta_2<2\pi$.
\end{itemize}
 First we introduce the bicomplex generalized states in the form
\begin{eqnarray}\label{eq:csj}
\mid Z>= \frac{1}{\sqrt{\mathscr{N}_{\mathscr{V}}(\mid Z\mid_h^2)}}\sum_{n=0}^\infty \frac{Z^n}{\sqrt{\rho(n)}}\mid n>,
\end{eqnarray}
where the corresponding normalization functions takes the following form
\begin{eqnarray}\label{cs3}
  \mathscr{N}_{\mathscr{V}}(Y)&=& (-Y)^{\frac{-\mathcal{V}}{2}}\Gamma_b(\mathcal{V}+1)\mathcal{J}_\mathcal{V}(i\sqrt{Y})\nonumber\\
  &=& \sum_{l=1}^{2}\sum_{n_l=0}^\infty\frac{\Gamma(\nu_l+1)y_l^{n_l}}{4^{n_l}n_l!\Gamma(\nu_l+n_l+1)}e_l,\;\;Y=\sum_{l=1}^2y_le_l=|Z|_h^2=\sum_{l=1}^2|z_l|^2e_l\nonumber
\end{eqnarray}
 and
 \begin{align}
 \rho(n)=\sum_{l=1}^2\frac{4^{n_l}\Gamma(n_l+1)\Gamma(\nu_l+n_l+1)}{\Gamma(\nu_l+1)}e_l.
 \end{align}
 Now, $\rho(n)$ satisfies the following recurrence relation
 \begin{align}\label{eq:cs9}
     \rho(n+1)&=\sum_{l=1}^2\frac{4^{n_l+1}\Gamma(n_l+2)\Gamma(\nu_l+n_l+2)}{\Gamma(\nu_l+1)}e_l\nonumber\\
    &=4\left[\sum_{l=1}^2(n_l+1)(\nu_l+n_l+1)e_l\right]\rho(n) \;\; \mbox{and} \\ \rho(0)&=1.\nonumber
 \end{align}
 For the general approach to the construction mentioned above $\rho(n)$ must be a positive hyperbolic number and the restriction imposed on $\mathcal{V}$ is that $\mathcal{V}$ is a hyperbolic number and satisfies $\mathcal{V}>_h-1$. We can now evaluate the scalar product by utilizing the normalized function as
 \begin{align}\label{eq:chu}
     <Z \mid Z'>=\frac{\mathscr{N}_{\mathscr{V}}(Z^*Z')}{\sqrt{\mathscr{N}_{\mathscr{V}}(\mid Z\mid_h^2)}\sqrt{\mathscr{N}_{\mathscr{V}}(\mid Z'\mid_h^2)}},
 \end{align}
 which is normalized but not orthogonal. Replacing, $Z$ by $i\sqrt{Y}$ in Theorem \ref{eq:s3} we obtain the hyperbolic radius of convergence of $\mathscr{N}(Y)$ is infinite, so the expression $<Z \mid Z'>$ is well define. Let us consider
 \begin{align}
     f(r)=\sum_{l=1}^2\sqrt{4(r_l+1)(\nu_l+r_l+1)}e_l
 \end{align}
 where $r=\sum\limits_{l=1}^2r_le_l$. Then
 \begin{align}
     \prod_{r=0}^{n-1}f(r)=\sum_{l=1}^2\sqrt{\frac{4^{n_l}\Gamma(n_l+1)\Gamma(\nu_l+r_l+1)}{\Gamma(\nu_l+1)}}=\sqrt{\rho(n)}.
 \end{align}
 Now we define bicomplex generalized annihilation and creation operators as
 \begin{eqnarray}
\mathcal{A}_-&=&\sum_{n=0}^\infty f(n)\mid n><n+1\mid\label{csn1}\nonumber\\
\mathcal{A}_+&=&\sum_{n=0}^\infty f(n)\mid n+1><n\mid\label{csn2}\nonumber
\end{eqnarray}
where $\mathcal{A}_+$ is adjoint operator of $\mathcal{A}_-$. Following the same approach as in \cite{bicomplex ghf}, we obtain the annihilation $\mathcal{A}_-$, creation operators $\mathcal{A}_+$ generate the bicomplex coherent states and these operators satisfies the following relation:
\begin{equation}\label{eqa}
\mathcal{A}_-\mid n>=f(n-1)\mid n-1>\nonumber
\end{equation}
and
\begin{equation}\label{eqc}
\mathcal{A}_+\mid n>=f(n)\mid n+1>.
\end{equation}
Therefore, the operator $\mathcal{A}_-$ acts as a lowering operator, whereas its conjugate operator $\mathcal{A}_+$  serves as a raising operator. Their product operator in the normal ordered manner are diagonal operators in the basis of Fock states
\begin{eqnarray}\label{csn7}
\mathcal{A}_-\mathcal{A}_+\mid n>&=&\left[f(n)\right]^2\mid n>\nonumber\\
\Rightarrow <n\mid \mathcal{A}_-\mathcal{A}_+\mid n>&=&\left[f(n)\right]^2\nonumber
\end{eqnarray}
and
\begin{eqnarray}\label{csn8}
\mathcal{A}_+\mathcal{A}_-\mid n>&=&\left[f(n-1)\right]^2\mid n>\nonumber\\
\Rightarrow <n\mid \mathcal{A}_+\mathcal{A}_-\mid n>&=&\left[f(n-1)\right]^2.\nonumber
\end{eqnarray}
These two operators $\mathcal{A}_-$ and $\mathcal{A}_+$ are non-commutative, their commutator is given by
\begin{equation}\label{csn9}
\left[\mathcal{A}_-,\mathcal{A}_+\right]=\sum_{n=0}^{\infty}\left(\left[f(n)\right]^2-\left[f(n-1)\right]^2\right)\mid n><n\mid,\nonumber
\end{equation}
and by applying the recurrence relation \eqref{eq:cs9}, we readily find that the generalized states $\mid Z>$ are eigenstates of annihilation operator $\mathcal{A}_-$, with bicomplex eigen value $Z$ and those states are coherent states.
Let us take the ground state $\mid 0>=\mid 0>e_1+\mid 0>e_2$, such that the lowering operator acts in the following manner $\mathcal{A}_- \mid 0>=0\mid 0>e_1+0\mid 0>e_2$. Now using \eqref{eqc}, we obtain
\begin{align*}
    (\mathcal{A}_+)^n\mid 0>&=\prod_{r=0}^{n-1}f(r)\sum_{l=1}^2\mid n_l>e_l\\
    &=\sqrt{\rho(n)}\mid n>.
\end{align*}
Since, $\mathcal{A}_+$ is the adjoint operator of $\mathcal{A}_-$, the following relations hold
\begin{align}\label{eq:csb}
    \mid n>=\frac{1}{\sqrt{\rho(n)}}(\mathcal{A}_+)^n\mid 0>\;\; \mbox{and} \;\;< n \mid=\frac{1}{\sqrt{\rho(n)}} <0 \mid (\mathcal{A}_-)^n.
\end{align}
Using \eqref{eq:csb} and \eqref{eq:csj}, we get
\begin{align}\label{eq:csn}
    \mid Z>= \frac{1}{\sqrt{\mathscr{N}_{\mathscr{V}}(\mid Z\mid_h^2)}}\sum_{n=0}^\infty \frac{(Z\mathcal{A}_+)^n}{{\rho(n)}}\mid 0>\;\;\mbox{and}\;\;<Z \mid= \frac{1}{\sqrt{\mathscr{N}_{\mathscr{V}}(\mid Z\mid_h^2)}}<0 \mid\sum_{n=0}^\infty \frac{(Z^\ast\mathcal{A}_-)^n}{{\rho(n)}}.
\end{align}
Applying the DOOT method \cite{cs_doot} along with the completeness relation of Fock states, we derive the projector
 $\mid 0><0 \mid$ corresponding to the ground state $\mid 0>$
 \begin{align}\label{csh}
     &\sum_{n=0}^{\infty}\mid n><n\mid=1\nonumber\\
    &\Rightarrow\sum_{n=0}^{\infty}\frac{1}{\rho(n)}\#(\mathcal{A}_+)^n\mid 0><0 \mid (\mathcal{A}_-)^n\#=1\nonumber\\
    &\Rightarrow \mid 0><0 \mid \sum_{l=1}^2\sum_{n_l=0}^{\infty}\frac{\Gamma(\nu_l+1)e_l}{4^{n_l}n_l!\Gamma(\nu_l+n_l+1)}\#  (\mathcal{A}_+\mathcal{A}_-)^{n_l} \#=1\nonumber\\
    &\Rightarrow \mid 0><0 \mid \#\mathscr{N}_{\mathscr{V}}(\mathcal{A}_+\mathcal{A}_-)\#=1\nonumber\\
    &\Rightarrow \mid 0><0 \mid=\frac{1}{\#\mathscr{N}_{\mathscr{V}}(\mathcal{A}_+\mathcal{A}_-)\#}=\frac{i^{\mathcal{V}}}{\Gamma_b(\mathcal{V}+1)}\#\frac{(\mathcal{A}_+\mathcal{A}_-)^{\frac{\mathcal{V}}{2}}}{\mathcal{J}_\mathcal{V}(i\sqrt{\mathcal{A}_+\mathcal{A}_-})}\#.
      \end{align}
From \eqref{eq:chu}, we deduce the continuity property
\begin{align*}
    &\lim_{Z'\rightarrow Z} ||\mid Z'>-\mid Z>||^2\\
    &=\lim_{Z'\rightarrow Z}\left[<Z' \mid Z'>+<Z \mid Z>-<Z' \mid Z>-<Z \mid Z'>\right]\\
    &=0.
\end{align*}
Now, we have determine the weight function $\mathcal{W}(|Z|_h)$ such the bicomplex generalized coherent state fulfill the resolution of unity
\begin{align}\label{eq:csm}
    &\int d \mu(Z)\mid Z><Z\mid=1.
\end{align}
Using \eqref{eq:csn},\eqref{csh} and \eqref{eq:csm}, we get
\begin{align}\label{eq:csl}
    &\int d \mu(Z) \frac{1}{\mathscr{N}_{\mathscr{V}}(\mid Z\mid_h^2)}\sum_{n=0}^\infty \frac{\#(Z\mathcal{A}_+)^n(Z^\ast\mathcal{A}_-)^n\#}{{[\rho(n)]^2}}\mid 0><0\mid=1\nonumber\\
    &\Rightarrow \int \frac{d\psi_Z}{2\pi}d(|Z|^2_h)\mathcal{W}(|Z|_h)\frac{1}{\mathscr{N}_{\mathscr{V}}(\mid Z\mid_h^2)}\sum_{n=0}^\infty \frac{\#(Z\mathcal{A}_+)^n(Z^\ast\mathcal{A}_-)^n\#}{{[\rho(n)]^2}}=\frac{1}{\mid 0><0\mid}\nonumber\\
    &\Rightarrow \left[\sum_{l=0}^2\int_{0}^{2\pi} \frac{d\theta_l}{2\pi}e_l\right]\;\int_{0}^{\infty}d(|Z|^2_h)\frac{\mathcal{W}(|Z|_h)}{\mathscr{N}_{\mathscr{V}}\left(\mid Z\mid_h^2\right)}\sum_{n=0}^\infty \frac{(|Z|_h^2)^n}{{[\rho(n)]^2}}\#(\mathcal{A}_+\mathcal{A}_-)^n\#=\#\mathscr{N}_{\mathscr{V}}(\mathcal{A}_+\mathcal{A}_-)\#\nonumber\\
    &\Rightarrow \sum_{n=0}^\infty \frac{\#(\mathcal{A}_+\mathcal{A}_-)^n\#}{{[\rho(n)]^2}}\int_{0}^{\infty}d(|Z|^2_h)\frac{\mathcal{W}(|Z|_h)}{\mathscr{N}_{\mathscr{V}}\left(\mid Z\mid_h^2\right)}(|Z|_h^2)^n=\#\mathscr{N}_{\mathscr{V}}(\mathcal{A}_+\mathcal{A}_-)\#.
    \end{align}
    Let us consider $\tilde{\mathcal{W}}(|Z|_h)=\frac{\mathcal{W}(|Z|_h)}{\mathscr{N}_{\mathscr{V}}\left(\mid Z\mid_h^2\right)}=\sum\limits_{l=1}^2W_{l}e_l$ and substitute in \eqref{eq:csl}, we get
    \begin{align}\label{eq:cs2}
    & \sum_{n=0}^\infty \frac{\#(\mathcal{A}_+\mathcal{A}_-)^n\#}{{\rho(n)}}\left[\frac{1}{\rho(n)}\int_{0}^{\infty}d(|Z|^2_h)\tilde{\mathcal{W}}(|Z|_h)(|Z|_h^2)^n\right]=\#\mathscr{N}_{\mathscr{V}}(\mathcal{A}_+\mathcal{A}_-)\#\nonumber\\
    &\Rightarrow \sum_{l=1}^2\sum_{n_l=0}^{\infty}\frac{\Gamma(\nu_l+1)}{4^{n_l}\Gamma(n_l+1)\Gamma(\nu_l+n_l+1)}\#  (\mathcal{A}_+\mathcal{A}_-)^{n_l} \# \nonumber\\& \hspace{55pt}\left[\frac{\Gamma(\nu_l+1)}{4^{n_l}\Gamma(n_l+1)\Gamma(\nu_l+n_l+1)}\int_{0}^{\infty}d(y_l)W_l\;(y_l)^{n_l}\right]e_l=\#\mathscr{N}_{\mathscr{V}}(\mathcal{A}_+\mathcal{A}_-)\#\nonumber\\
    &\Rightarrow \int_{0}^{\infty}d(y_l)W_l\;(y_l)^{n_l}=\frac{4^{n_l}\Gamma(n_l+1)\Gamma(\nu_l+n_l+1)}{\Gamma(\nu_l+1)},\;\; \mbox{for $l=1,2$}.
    \end{align}
  Substituting
$n_l=s-1$ into equation \eqref{eq:cs2}, yields an integral that represents a Stieltjes moment problem (\cite{cs_dis_con, cs_st_ha}) and applying the standard formula involving the classical integral of a Meijer G-function
  \begin{align*}
      \int_{0}^\infty x^{s-1}
G_{p, q}^{\,m, n} \left( \lambda x \;\middle|\; \begin{matrix} \alpha_1, \ldots, \alpha_p \\ \beta_1, \ldots, \beta_q \end{matrix} \right)dx=\frac{\prod\limits_{i=1}^{m}\Gamma(\beta_j+s)\prod\limits_{i=1}^{n}\Gamma(1-\alpha_i-s)}{\lambda^s\prod\limits_{j=m+1}^q\Gamma(1-\beta_j-s)\prod\limits_{i=n+1}^{p}\Gamma(\alpha_i+s)}
  \end{align*}
in \eqref{eq:cs2}, we obtain
\begin{align*}
    W_l&=\frac{1}{4\Gamma(\nu_l+1)}\;G_{0, 2}^{\,2, 0} \left( \frac{y_l}{4} \;\middle|\; \begin{matrix} - \\ 0, \nu_l \end{matrix} \right)\\
    &=\frac{1}{4\Gamma(\nu_l+1)}\;G_{0, 2}^{\,2, 0} \left( \frac{|z_l|^2}{4} \;\middle|\; \begin{matrix} - \\ 0, \nu_l \end{matrix} \right)>0,
\end{align*}
 for $\nu_l>-1,l=1,2$ and the weight function is given by
\begin{align*}
    \mathcal{W}(|Z|_h)&=\mathscr{N}_{\mathscr{V}}\left(\mid Z\mid_h^2\right)\sum_{l=1}^2W_le_l\\
    &=\frac{\left(i|Z|_h\right)^{-\mathcal{V}}\Gamma_b(\mathcal{V}+1)\mathcal{J}_\mathcal{V}(i|Z|_h)}{4\Gamma_b(\mathcal{V}+1)}\sum_{l=1}^2G_{0, 2}^{\,2, 0} \left( \frac{|z_l|^2}{4} \;\middle|\; \begin{matrix} - \\ 0, \nu_l \end{matrix} \right)e_l>_h0,\;\; \mbox{for $\mathcal{V}>_h-1$}.
\end{align*}
Therefore, the corresponding integration measure is given by
\begin{align*}
    d\mu(Z)=2\frac{d\psi_Z}{\pi}d(|Z|^2_h)\left(i|Z|_h\right)^{-\mathcal{V}}\Gamma_b(\mathcal{V}+1)\mathcal{J}_\mathcal{V}(i|Z|_h)\sum_{l=1}^2G_{0, 2}^{\,2, 0} \left( \frac{|z_l|^2}{4} \;\middle|\; \begin{matrix} - \\ 0, \nu_l \end{matrix} \right)e_l.
\end{align*}
Hence, the bicomplex generalized states \eqref{eq:csj} fulfill key properties such as normalizability, continuity, and resolution of unity. These features are highly valuable because they can be used in the future for potential applications in areas such as quantum optics, nonlinear systems, quantum information, and signal processing.
\section{Conclusion}
In this paper, we extend the Bessel function to bicomplex space and explore its properties, including recurrence relations, differential relations and integral representations. Moreover, we examine bicomplex holomorphicity and analyze its asymptotic behavior. Furthermore, we introduce the $n$-dimensional bicomplex Hankel transform and establish its properties as a generalization of the n-dimensional Hankel transform \cite{hankel}. In Example \ref{ex1} and Example \ref{ex5}, we analyze the solutions of partial differential equations involving the operator $\mathcal{M}_{\mathcal{V}}\mathcal{N}_{\mathcal{V}}$ through the use of the $n$-dimensional bicomplex Hankel transform. Additionally, in Remark \ref{re2} and Remark \ref{re6}, we present a novel bicomplex generalization of the classical wave and heat equations and derive their respective solutions. Moreover, we provide graphical representations of these solutions in Figure~\ref{fig:wave} and Figure~\ref{fig:heat} respectively. In the last section, we present a new application where the generalized coherent states is derived by using bicomplex Bessel function and this state is shown to satisfy the essential properties of normalizability, continuity, and resolution of the identity. The construction of coherent states using bicomplex Bessel functions is still unreported. In future, the $n$-dimensional bicomplex Hankel transform can be applied in areas such as optics, signal processing, quantum mechanics, electromagnetic theory and various other problems in mathematical physics and engineering.


\end{document}